\documentclass[reqno]{amsart}

\usepackage{amsmath,amssymb,mathrsfs,stmaryrd,enumitem}
\newtheorem{thm}{Theorem} 

\setlist[enumerate]{label = (\alph*)}

\theoremstyle{definition}
\newtheorem{theorem}{Theorem}[section]

\newtheorem{corollary}[theorem]{Corollary} 
\newtheorem{definition}[theorem]{Definition} 
\newtheorem{example}[theorem]{Example}
\newtheorem{lemma}[theorem]{Lemma}
\newtheorem{proposition}[theorem]{Proposition}
\newtheorem{question}[theorem]{Question}
\newtheorem{remark}[theorem]{Remark}

\DeclareMathOperator\Aut{Aut}
\DeclareMathOperator\Autp{Aut_P}
\DeclareMathOperator\ch{char}
\DeclareMathOperator\der{Der}
\DeclareMathOperator\ex{exp}
\DeclareMathOperator\id{id}
\DeclareMathOperator\Maxspec{MaxSpec}
\DeclareMathOperator\GKdim{GKdim}
\DeclareMathOperator\GrAutp{GrAut_P}
\DeclareMathOperator\Span{Span}
\DeclareMathOperator\Spec{Spec}
\DeclareMathOperator\supp{supp}
\DeclareMathOperator\trace{trace}
\DeclareMathOperator\Tr{Tr}

\newcommand\cnt{\mathcal Z}
\newcommand\GrAutpw{\mathrm{GrAut}_\mathrm{P}^w}
\newcommand\inv{^{-1}}
\newcommand\iso{\cong}
\newcommand\kk{\Bbbk}
\newcommand\niso{\ncong}
\newcommand\pcnt{\cnt_P}
\newcommand\pder{\mathrm{PDer}}
\newcommand\tornado{\xi}

\newcommand\cL{\mathcal L}
\newcommand\cN{\mathcal N}
\newcommand\cO{\mathcal O}
\newcommand\cP{\mathcal P}

\newcommand\scP{\mathscr P}
\newcommand\scH{\mathscr H}

\newcommand\NN{\mathbb N}
\newcommand\ZZ{\mathbb Z}

\newcommand\fg{\mathfrak g}
\newcommand\fm{\mathfrak m}

\newcommand\bp{\mathbf p}
\newcommand\bx{\mathbf x}
\newcommand\by{\mathbf y}

\newcommand\bI{\mathbf I}
\newcommand\bJ{\mathbf J}
\newcommand\bK{\mathbf K}

\newcommand\grp[1]{{\langle #1 \rangle}}

\newcommand\restr[2]{{
  \left.\kern-\nulldelimiterspace 
  #1 
  \vphantom{\big|} 
  \right|_{#2} 
  }}

\begin{document}

\addtocontents{toc}{\setcounter{tocdepth}{1}} 

\title{Reflection groups and rigidity of quadratic Poisson algebras}

\author[Gaddis]{Jason Gaddis}
\address{Department of Mathematics, Miami University, Oxford, Ohio 45056} 
\email{gaddisj@miamioh.edu}

\author[Veerapen]{Padmini Veerapen}
\address{Department of Mathematics,Tennessee Technological University, Cookeville, Tennessee 38505}
\email{pveerapen@tntech.edu}

\author[Wang]{Xingting Wang}
\address{Department of Mathematics, Howard University, Washington DC, 20059} 
\email{xingting.wang@howard.edu}


\begin{abstract}
In this paper, we study the invariant theory of quadratic Poisson algebras. Let $G$ be a finite group of the graded Poisson automorphisms of a quadratic Poisson algebra $A$. When the Poisson bracket of $A$ is skew-symmetric, a Poisson version of the Shephard-Todd-Chevalley theorem is proved stating that the fixed Poisson subring $A^G$ is skew-symmetric if and only if $G$ is generated by reflections. For many other well-known families of quadratic Poisson algebras, we show that $G$ contains limited or even no reflections. This kind of Poisson rigidity result ensures that the corresponding fixed Poisson subring $A^G$ is not isomorphic to $A$ as Poisson algebras unless $G$ is trivial. 
\end{abstract}

\maketitle

\section*{Introduction}

Group actions are ubiquitous in mathematics and theoretical physics. To study the symmetry of an algebraic object it is often useful to understand what groups act on it. Invariants of the action of a finite group on a commutative polynomial ring have played a major role in the development of commutative algebra (see e.g., \cite{NS02}). One of the earliest results in classical invariant theory, attributed to Gauss, says that the subring of invariants of the polynomial ring in $n$ variables under the action of the symmetric group $S_n$ is generated by the $n$ elementary symmetric polynomials. Note since the symmetric polynomials are algebraically independent, the subring of invariants is also a polynomial ring. Gauss's result paved the way for the following famous theorem.

\begin{thm}[{\bf Shephard-Todd-Chevalley Theorem}]
\label{thm.clstc} For $\kk$ a field of characteristic zero, the subring of invariants $\kk[x_1, \dots, x_n]^G$ under a finite group $G$ is a polynomial ring if and only if $G$ is generated by reflections.  
\end{thm}

Borel's history \cite[Chapter VII]{Bo01} provides details on the origins of the ``Shephard-Todd-Chevalley Theorem".  It was solved in 1954 for $\kk$ an algebraically closed field of characteristic zero by Shephard and Todd \cite{ShTo}, who classified the complex reflection  groups and produced their invariants.  Shortly afterward, Chevalley \cite{Chev} gave a more abstract argument, which showed that for real reflection groups $G$, the fixed subring $\kk[x_1, \dots, x_n]^G$ is a polynomial ring, and Serre \cite{Se68} showed that Chevalley's argument could be used to prove the result for all unitary reflection groups.

The notion of a Poisson bracket was introduced by French mathematician Sim\'eon Denis Poisson in the search for integrals of motion in Hamiltonian mechanics \cite{pstructures}. Recently, Poisson algebras have become deeply entangled with noncommutative geometry, integrable systems, topological field theories, and representation theory of noncommutative algebras \cite{BG2,BZuni,LT20,pol1,WWY2,WWY1}. There is an intricate connection between quadratic Poisson algebras and quantum polynomial rings through deformation theory and semiclassical limits. For example, it is conjectured that the primitive ideal space of the quantized coordinate ring  $\mathcal O_q(G)$ of a semisimple Lie group $G$ and the Poisson primitive ideal space of the classical coordinate ring $\mathcal O(G)$, with their respective Zariski topologies, are homeomorphic (see e.g., \cite{Good10, LT20}). In order to study the symmetry involved in Poisson brackets, we propose in this paper the study of group actions on Poisson algebras. 

Let $\kk$ be a base field which is algebraically closed of characteristic zero. In this paper, we say a Poisson structure on $A=\kk[x_1,\hdots,x_n]$ is \emph{quadratic} if $\{x_i,x_j\} \in A_2$ for all $i,j$ where $A$ is equipped with the standard grading, that is, $A_1=\kk x_1\oplus \cdots \oplus \kk x_n$. This class includes many Poisson algebras which appear as semiclassical limits of families of quadratic noncommutative algebras. For example, the skew-symmetric Poisson algebras satisfying $\{x_i,x_j\}=q_{ij}x_ix_j$ for some skew-symmetric scalar matrix $(q_{ij})_{1\le i,j\le n}$ are semiclassical limits of skew polynomial rings. We say a (classical) reflection of $A$ 
is a \emph{Poisson reflection} if it is also a Poisson automorphism. 
Background material is reviewed in Section \ref{sec:bg}, and in Section \ref{sec.quadratic} we develop tools to study fixed subrings of Poisson algebras.

In Section \ref{sec.skew}, we investigate the Shephard-Todd-Chevalley (STC) theorem for quadratic Poisson algebras.  In some sense, the problem in the Poisson setting might seem trivial. For any polynomial Poisson algebra $A$ and a finite subgroup $G$ of all Poisson graded automorphisms of $A$, the fixed Poisson subring $A^G$ is again a Poisson polynomial ring when $G$ is generated by (classical) reflections. That is, the problem reduces to a direct application of the STC theorem. However, inspired by the work of \cite{KKZ3,KKZ1} in the noncommutative algebra setting, our goal is to explore specific Poisson structures - discussed below - and determine whether these Poisson structures are preserved under the action of certain graded Poisson automorphisms of $A$. The following result is a Poisson analogue of \cite[Theorem 5.5]{KKZ1}.

\begin{thm}[Theorem \ref{thm.fixed}]
\label{thm.stc}
Let $A=\kk[x_1,\hdots,x_n]$ be a skew-symmetric Poisson algebra and let $G$ be a finite subgroup of graded Poisson automorphisms of $A$. Then $A^G$ has skew-symmetric Poisson structure if and only if $G$ is generated by Poisson reflections.
\end{thm}

The notion of rigidity has been investigated for various families of noncommutative algebras (see e.g., \cite{AP,KKZ3,Sm1,Tik19}).  In this paper, we introduce our notion of rigidity for Poisson algebras under the actions of any finite Poisson automorphism subgroups.  

\begin{definition}[Poisson Rigidity]
Let $A$ be a (graded) Poisson algebra and let $G$ be a finite subgroup  of the (graded) Poisson automorphisms of $A$. We say $A$ is \emph{(graded) rigid} if $A^G\cong A$ as Poisson algebras implies that $G$ is trivial.
\end{definition}

 As we will show in the upcoming sections, in contrast to the STC theorem, it is possible for a Poisson algebra to be graded rigid and yet for the group $G$ of automorphims to contain nontrivial Poisson reflections (Lemma \ref{lem.jac2} and Example \ref{ex.mat2}). Another extreme case occurs when $A$ contains no Poisson reflections (Lemma \ref{lem.jac1}, Theorem \ref{thm.mat}, and Theorem \ref{thm.kkbracket}(b)). By the (classical) STC theorem, $A^G$ is not a polynomial Poisson algebra unless $G$ is trivial, so $A$ is graded rigid. As an example, we prove our first rigidity result for quadratic Poisson algebras $A=\kk[x,y]$ with nonzero Poisson brackets showing that they are all graded rigid (Proposition \ref{prop.n2}). We will prove the rigidity results for various classes of quadratic Poisson algebras in Section \ref{sec.rigidity}.

\begin{thm}\label{thm.m2}
The following quadratic Poisson algebras are graded rigid. 
\begin{enumerate}
    \item (Theorem \ref{thm.jac}) The Poisson algebra $A=\kk[x,y,z]$ with Jacobian Poisson bracket determined by some nonzero potential $f_{p,q} = \frac{p}{3}(x^3+y^3+z^3)+qxyz$ with $p,q \in \kk$.
    \item (Theorem \ref{thm.rigM}) The coordinate ring of $n \times n$ matrices $\cO(M_n)$ with Poisson structure given by the semiclassical limit of the quantum $n\times n$ matrices for any $n\ge 2$.
    \item (Theorem \ref{thm.rigidHWPA}) The homogenization $\scH_n$ of the $n$th Weyl Poisson algebra $\scP_n$ for any $n\ge 1$.
    \item (Theorem \ref{thm.kkbracket}) The homogenization $PH(\mathfrak g)$ of the Kostant-Kirillov bracket on the  symmetric algebra $S(\mathfrak g)$ for any finite dimensional Lie algebra $\fg$ with no $1$-dimensional Lie ideal.
\end{enumerate}
\end{thm}
In particular, for case (b) when $n>2$ and case (d), we are able to show that there are no Poisson reflections and hence graded rigidity follows immediately by the STC theorem. For other cases, we show that there are strong restrictions on the Poisson reflections with respect to the corresponding Poisson brackets. The Poisson reflections in these cases are easy to classify and by a direction computation their fixed Poisson subrings are not isomorphic to a Poisson algebra of the same type.

One of the earliest results in noncommutative invariant theory was the rigidity of the first Weyl algebra $A_1$, proved by Smith \cite{Sm1}. It says that for any finite subgroup $G$ of $\Aut(A_1)$, an isomorphism $A_1^G \cong A_1$ implies that $G$ must be trivial. Later the result was extended to the $n$th Weyl algebra $A_n$ by Alev and Polo \cite{AP}. In our paper, we show that the $n$th Weyl Poisson algebra $\scP_n$ is rigid by using a result of Tikaradze regarding the derived invariants of the fixed ring of enveloping algebras of semisimple Lie algebras and its analogue version for Poisson varieties (Theorem \ref{thm.rigW}). This reveals a connection between the rigidity of the $n$th Weyl Poisson algebra $\scP_n$ and a conjecture of Belov-Kanel and Kontsevich stating that the Poisson automorphism group of $\scP_n$ and the automorphism group of the $n$th Weyl algebra $A_n$ are naturally isomorphic \cite{BK05}.

In Section \ref{sec.UEA}, we examine the interplay between quantum polynomial rings and quadratic Poisson algebras by exploring the Poisson universal enveloping algebras. In noncommutative invariant theory, there is a notion of a quasi-reflection on quantum polynomial rings introduced in \cite{KKZ3} as a noncommutative analogue of reflections on polynomial rings. If $A=\kk[x_1,\hdots,x_n]$ is a quadratic Poisson algebra, then the Poisson universal enveloping algebra, denoted by $U(A)$, is a quantum polynomial ring (Lemma \ref{lem.qpr}). In particular, a (graded) Poisson automorphism of $A$ naturally extends to an (graded) automorphism of $U(A)$.  However, this relationship does not yield information on quasi-reflections of $U(A)$.

\begin{thm}[Theorem \ref{thm.qpr}]
If $A$ is a quadratic Poisson algebra and $g$ a Poisson reflection of $A$, then $g$ is not a quasi-reflection of $U(A)$.
\end{thm}

Finally in Section \ref{sect:rm}, we give some remarks and propose some questions as a continuation of our project on invariant theory of quadratic Poisson algebras. 

\tableofcontents

\subsection*{Acknowledgements}
The authors thank Akaki Tikaradze for pointing out his paper \cite{TIK2}, which resolved one of the questions in an early draft of this paper, and his willingingess to allow the inclusion of Remark \ref{rem.tik}. The authors also appreciate the referee's close reading and suggestions for improving the manuscript. The third author was partially supported by Simons Collaboration Grant No. 688403

\section{Background on Poisson algebras and invariant theory}
\label{sec:bg}
Throughout the paper, we work over an algebraically closed base field $\kk$ of characteristic zero. 

\subsection{Classical and noncommutative invariant theory}
Let $V$ be a finite dimensional vector space over $\kk$ and let $g$ be a linear operator on $V$ of finite order. We call $g$ a \emph{reflection} on $V$ if $\dim V^{\grp{g}}=\dim V-1$, where $V^{\grp{g}}$ is the $g$-invariant subspace of $V$. Stated another way, $g$ is a reflection if and only if all but one of the eigenvalues of $g$ are $1$. We note that, in \cite{benson}, the term ``reflection" is reserved for the case that $g$ has real eigenvalues (so that the single non-identity eigenvalue is $-1$) and the term ``pseudo-reflection" is used for the case that the single non-identity eigenvalue is a complex root of unity. Throughout this paper, we use the term ``reflection" to refer to a both reflections and pseudo-reflections.

For a more general approach adaptable to the noncommutative setting, let $R=\bigoplus_{i\geq 0} R_i$ be an $\NN$-graded locally finite $\kk$-algebra and let $g$ be a graded automorphism of $R$. Then $g$ is a linear operator on each homogeneous component $R_i$ of $R$. The \emph{trace series} of $g$ is defined as
\begin{align}
\label{eq.trace}
\Tr_R(g,t)=\sum_{i=0}^\infty \trace\left( \restr{g}{R_i} \right)\,t^i\in \kk\llbracket t\rrbracket.
\end{align}
In particular,
$\Tr_R(\id_R,t)=\sum_{i\ge 0} \dim R_i\,t^i=H_R(t)\in \mathbb Z\llbracket t\rrbracket$, which is the Hilbert series of $R$. 
When $R=\kk[V]$ as above, we may represent $g$ as a matrix by its action on $V$ and one has
\[ \Tr_{\kk[V]}(g,t) = \frac{1}{\det(1-g\inv t,V)}.\]
We refer the reader to \cite[\S 2.5-2.6 and \S 3.7 ]{benson} for a thorough account of the trace series in the commutative setting. 

If $g$ is a reflection of $V$ of finite order, then it follows directly from the definition that there is a basis $\{y_1,\hdots,y_n\}$ of $V$ such that $g(y_1)=\tornado y_1$ for some root of unity $\tornado \neq 1$ and $g(y_i)=y_i$ for $i>1$. Let $R=\kk[V]$, then we have
\begin{align}
\label{eq.refl}
\Tr_R(g,t) =  \frac{1}{(1-\tornado t)(1-t)^{n-1}}.
\end{align}
However, the above definition is not suitable in the non-commutative setting, see \cite{KKZ1} for examples illustrating this fact. 

As in \cite{KKZ3}, we define a \emph{quantum polynomial ring (of dimension $n$)} $R$ to be a noetherian connected ($\NN$-)graded algebra of global dimension $n$ with Hilbert series $H_R(t) = (1-t)^{-n}$. Examples of quantum polynomial rings include quantum affine spaces and quantum matrix algebras. In this case, we say $g$ is a \emph{quasi-reflection} of $R$ if $\Tr_R(g,t)$ satisfies \eqref{eq.refl} for some root of unity $\tornado \neq 1$.
Thus, in the commutative case when $R=\kk[V]$, $g$ is a quasi-reflection if and only if it is a reflection \cite[Lemma 2.1]{KKZ3}. 

\subsection{Background on Poisson algebras}
\label{sec.background}

A \emph{Poisson algebra} is a commutative $\kk$-algebra $A$ equipped with a $\kk$-bilinear bracket $\{-,-\}: A\times A\to A$ such that $(A,\{-,-\})$ is a Lie algebra and 
\[ \{ a, bc\} = b\{a,c\} + \{a,b\}c,\]
for all $a,b,c \in A$. The \emph{Poisson center} of $A$ is the set $\pcnt(A) = \{ a \in A : \{a,b\} = 0 \text{ for all } b \in A\}$.
An algebra automorphism $\phi$ of $A$ is called a \emph{Poisson automorphism}
if $\phi(\{a,b\}) = \{ \phi(a),\phi(b) \}$ for all $a,b \in A$. We denote the group of all Poisson automorphisms of $A$ by $\Autp(A)$. If $G$ is a subgroup of $\Autp(A)$, then $A^G$ is naturally a Poisson subalgebra of $A$ \cite{AF}. This follows since for all $g \in G$ and all $a,b \in A^G$,
\[ g(\{a,b\}) = \{ g(a), g(b) \} = \{ a, b \}.\]

A Poisson algebra $A$ is said to be \emph{Poisson graded} by a monoid $(M,*)$ if $A = \bigoplus_{m \in M} A_m$ is a vector space decomposition of $A$ such that $A_m \cdot A_n \subset A_{m*n}$ and $\{A_m,A_n\} \subset A_{m*n}$. We refer to the elements of $A_m$ as homogeneous elements of degree $m$. In most of this paper, $M=\NN$. However, there are instances where $M=\ZZ^2$. 
Suppose $A$ is Poisson graded by $(M,*)$. A Poisson automorphism $g \in \Autp(A)$ is said to be \emph{graded} if it respects the grading above. That is, $g(A_m) = A_m$ for all $m \in M$. We denote the group of all graded Poisson automorphisms of $A$ by $\GrAutp(A)$ when there is no confusion about the grading. 

We say a Poisson algebra $A=\kk[x_1,\ldots,x_n]$ with the standard grading ($A_1=\kk x_1\oplus \cdots \oplus \kk x_n$) is \emph{quadratic} if $\{x_i,x_j\} \in A_2$ for all $i,j$. 
In this setting, the irrelevant ideal $A_{\geq 1}$ is a Poisson ideal of $A$. If $B=\kk[x_1,\hdots,x_m] \subset A$ is a Poisson subalgebra $(m\le n)$, we let $(B_{\geq 1})$ denote the Poisson ideal generated by $B_{\geq 1}$ in $A$.
A reflection $g\in GL_n(\kk)$ of $A$ is called a \emph{Poisson reflection of $A$} if $g\in \GrAutp(A)$. A subgroup $G$ of $\GrAutp(A)$ is called a \emph{Poisson reflection group} if it is generated by Poisson reflections of $A$. Let $G$ be a finite Poisson reflection group of $A$. Then it follows immediately that $A^G$ is a Poisson algebra whose underlying algebraic structure is isomorphic to $A$. In this work we will be primarily interested in what Poisson structures are (or are not) preserved under taking fixed rings.

\subsection{Unimodularity and Jacobian brackets}
A Poisson manifold is called \emph{unimodular} if the modular class vanishes or the modular vector field is a Hamiltonian vector field. Xu proved a duality between the Poisson homology and cohomology in the case of unimodular Poisson manifolds \cite{Xu}. Suppose that $A=\kk[x_1,\hdots,x_n]$ is a Poisson algebra with Poisson bracket $\{-,-\}$. If $\kk=\mathbb C$, then the affine space $\mathbb A^n$ corresponding to $A$ is a Poisson manifold together with the complex topology. In this case, the modular class is represented by the \emph{modular derivation} $\phi_\eta$ of $A$ given by 
\begin{equation}\label{equ.modular}
\phi_\eta(f):= \sum_{j=1}^n \frac{\partial\{f,x_j\}}{\partial x_j},
\end{equation}
for all $f \in A$ \cite[Lemma 2.4]{LWW}. Moreover, $\mathbb A^n$ is unimodular if and only if the modular derivation $\phi_\eta$ vanishes. Therefore, we say $A$ is \emph{unimodular} if $\phi_\eta=0$. By \cite[Theorem 5.8]{LWZ4}, the Poisson algebra $A$ is unimodular if and only if its Poisson universal enveloping algebra $U(A)$ is Calabi-Yau. (See Section \ref{sec.UEA} for the definition of $U(A)$.)  

\begin{example}\label{ex.unimodular}
We give some examples of unimodular Poisson algebras with nontrivial Poisson brackets. 
\begin{enumerate}
    \item Let $A=\kk[x,y]$ be a unimodular Poisson algebra. By \cite{LWW}, up to a possible Poisson isomorphism, $A$ has Poisson bracket $\{x,y\}=1$. That is, $A$ is the \emph{first Weyl Poisson algebra}.
    \item Let $A=\kk[x,y,z]$ be a Poisson algebra. We say the Poisson bracket on $A$ is \emph{Jacobian} (also called \emph{exact}) if it has the form
\[ 
\{x,y\} = \frac{\partial}{\partial z} f, \quad
\{y,z\} = \frac{\partial}{\partial x} f, \quad
\{z,x\} = \frac{\partial}{\partial y} f,
\]
for some nonzero element $0\neq f \in A$. In this case, $f$ is said to be the \emph{potential} associated to the bracket. Note that the Poisson algebra $A=\kk[x,y,z]$ is unimodular if and only if it has Jacobian bracket (\cite[Proposition 2.6]{LWW}, \cite{PRZ}).
\end{enumerate}
\end{example}

\subsection{Poisson normal elements}
\label{sec.PoissonNormal}

Let $A$ be a Poisson algebra. An element $f \in A$ is called \emph{Poisson normal} if $\{f,A\} \subset fA$. It is clear that $f$ is Poisson normal if and only if $(f)$ is a Poisson ideal in $A$.

\begin{example}\label{ex:normalskew}
Let $(\lambda_{ij})_{1\le i,j\le n}$ be some skew-symmetric matrix such that $\lambda_{ij}\neq 0$  for all $i\neq j$. Let $A$ be the quotient of a skew quadratic Poisson algebra $\kk_{\lambda_{ij}}[x_1,\hdots,x_n]/I$ with $\{x_i,x_j\}=\lambda_{ij}x_ix_j$ where $I$ is a graded Poisson ideal in $\kk_{\lambda_{ij}}[x_1, \hdots, x_n]_{\ge 3}$. By the proof of \cite[Theorem 4.6]{GXW}, the only Poisson normal elements in $A_1$ are scalar multiples of the variables $x_1,\hdots,x_n$. 
\end{example}

We now recall that a \emph{derivation} $\alpha$ of $A$ is a $\kk$-endormorphism of $A$ satisfying the Leibniz rule:
\[ \alpha(ab) = \alpha(a)b + a \alpha(b)\]
for all $a,b \in A$. Further, the derivation $\alpha$ is a \emph{Poisson derivation} of $A$ if it satisfies
\[ \alpha( \{ a,b \}) = \{ \alpha(a), b \} + \{ a, \alpha(b) \},\]
for all $a,b \in A$. The next lemma is a well-known result.

\begin{lemma}
\label{lem.pder}
Let $A$ be a Poisson algebra that is also an integral domain. If $f \in A$ is Poisson normal, then there is an associated Poisson derivation of $A$, denoted by $\pi_f$, such that for all $a\in A$,
\[ \{f,a\}=\pi_f(a)f. \]
\end{lemma}
\begin{proof}
It is clear from the definition of Poisson normality and the assumptions on $A$ that $\pi_f:A \to A$ is a well-defined $\kk$-linear map.

If $a,b \in A$, then by the Jacobi identity,
\[ \pi_f(ab)f = \{ f,ab\} = a \{ f,b\} + \{f,a\} b = (a\pi_f(b) + \pi_f(a)b)f.\]
Since $A$ is an integral domain, $\pi_f(ab) = a\pi_f(b) + \pi_f(a)b$. Similarly, by the definition of $\pi_f$ and the properties of the Poisson bracket,
\begin{align*}
\pi_f(\{a,b\})f 
	&= \{f, \{a,b\} \} 
	= \{ a, \{ f,b\} \} + \{ \{ f,a\}, b \} \\
	&= \{ a, \pi_f(b)f \} + \{ \pi_f(a)f, b \} \\
	&= \{ a, \pi_f(b)\} f + \{a,f\} \pi_f(b) + \{ \pi_f(a),b\} f + \pi_f(a) \{f,b\} \\
	&= \{ a, \pi_f(b)\} f - \pi_f(a)\pi_f(b)f + \{ \pi_f(a),b\} f + \pi_f(a)\pi_f(b)f \\
	&= ( \{ a, \pi_f(b)\} + \{ \pi_f(a),b\} )f.
\end{align*}
Again since $A$ is an integral domain and $\pi_f(\{a,b\}) = \{ a, \pi_f(b)\} + \{ \pi_f(a),b\}$.
\end{proof}

We denote by $\cN_P(A)$ the set of all Poisson normal elements in $A$. One can easily check that for $f,g \in \cN_P(A)$, $\pi_{fg} = \pi_{f+g}$ so that $\cN_P(A)$ is closed under the multiplication of $A$. 
Recall that the set of all Poisson derivations of $A$, denoted by $\pder(A)\subset \der(A)$, is a Lie subalgebra of all the derivations of $A$. By Lemma \ref{lem.pder}, we have a well-defined linear map $\pi: \cN_P(A)\to \pder(A)$ via $f\mapsto \pi_f$ for any $f\in  \mathcal N_P(A)$.

\subsection{Poisson Ore extensions}
Let $A$ be a Poisson algebra and let $\alpha$ be a Poisson derivation of $A$. The \emph{Poisson-Ore extension} $A[z;\alpha]_P$ is the polynomial ring $A[z]$ with Poisson bracket
\[ 
\{a,b\} = \{a,b\}_A, \qquad
\{z,a\} = \alpha(a)z\quad \text{for all } a,b \in A.
\]
We write $A[z]$ for $A[z;0]_P$. Poisson-Ore extensions were studied by Oh \cite{OH1}, but it seems that their origin dates back to Polishchuk \cite{pol1}.

\begin{example}
\label{ex.ore}
Let $A$ be a Poisson algebra and let $A[t;\alpha]=A[t;\alpha]_P$ be a Poisson-Ore extension of $A$ for some Poisson derivation $\alpha$ of $A$.
Let $\phi$ be a Poisson automorphism of $A[t;\alpha]$ such that $\phi(a)=a$ for all $a \in A$ and $\phi(t)=\tornado t$ for some primitive $m$th root of unity $\tornado$. If $G=\grp{\phi}$, then $A^G = A[t^m; m\alpha]$.
\end{example}

\section{Quadratic Poisson algebras}
\label{sec.quadratic}

In this section we develop the necessary tools for identifying Poisson reflections in quadratic Poisson algebras. In particular, we prove in Lemma \ref{lem.normal} that for a Poisson reflection $g$ of a quadratic Poisson algebra $A$, there is a basis of $A_1$ containing a Poisson normal element. These results are largely Poisson versions of results in \cite{KKZ3}.
In the second part of this section, we consider the case of quadratic Poisson structures on $A=\kk[x,y]$.


\begin{lemma}
\label{lem.Z2graded}
Suppose that $A=\kk[x_1,\hdots,x_n]$ is a quadratic Poisson algebra that is $\ZZ^2$-graded with $\deg x_i = (1,0)$ for $i=1,\hdots,m$ and $\deg x_j = (0,1)$ for $j=m+1,\hdots,n$ for some integer $1\le m\le n$.
Let $B$ and $C$ be graded subalgebras generated by
$\{x_1,\hdots,x_m\}$ and $\{x_{m+1},\hdots,x_n\}$, respectively.
\begin{enumerate}
\item Both $B$ and $C$ are quadratic Poisson subalgebras of $A$.
\item We have Poisson isomorphisms $B \iso A/(C_{\geq 1})$ and $C \iso A/(B_{\geq 1})$.
\item If $m=1$, then $A$ is the Poisson-Ore extension $C[x_1;\alpha]$ for some graded
Poisson derivation $\alpha$ of $C$. Consequently, $x_1$ is Poisson normal in $A$.
\end{enumerate}
\end{lemma}
\begin{proof}
We write $A_{(i,j)}$ for the degree $(i,j)$ homogeneous part of $A$. We abuse of notation and write $B=\bigoplus_{i\ge 0} B_{(i,0)}=\bigoplus_{i\ge 0} B_i$ and $C=\bigoplus_{i\ge 0}C_{(0,i)}=\bigoplus_{i\ge 0}C_i$ as graded subalgebras of $A$. 

(a) By a degree argument, we have $\{B_1,B_1\}\subseteq  A_{(2,0)}=B_2$. Since $B$ is generated by $B_1$, it follows that $B$ is a Poisson subalgebra of $A$. Similarly, $C$ is a Poisson subalgebra of $A$. 

(b) This follows analogously to \cite[Proposition 3.5(a)]{KKZ3}. There is a natural 
Poisson homomorphism $\phi: B \to A/(C_{\geq 1})$
induced from the composition $B \to A \to A/(C_{\geq 1})$. Clearly, $\phi$ is surjective.
If $x \in B$, then $\deg x \in \ZZ \times 0$. If $x \in (C_{\geq 1})$, then $\deg x \in \ZZ \times \ZZ^+$.
Hence, $B \cap A/(C_{\geq 1}) = 0$ and so $\phi$ is injective.
The second isomorphism follows similarly.

(c) By assumption, $A=C[x_1]$ as algebras. Moreover, since $\deg(\{x_1,x_i\}) = (1,1)$, for $i > 1$, then $\{x_1,x_i\}\in x_1C_1$, for $i>1$. The result now follows from Lemma \ref{lem.pder}.
\end{proof}

\begin{lemma}
\label{lem.normal}
Suppose that $A=\kk[x_1,\hdots,x_n]$ is a Poisson algebra and $g$ is a Poisson reflection of $A$.
Let $\{y_1,\hdots,y_n\} \subset A_1$ be a basis of $A$ with $g(y_1)=\tornado y_1$ for a root of unity $\tornado \neq 1$ and $g(y_i)=y_i$ for $2\le i\le n$ . The element $y_1$ is Poisson normal in $A$.
\end{lemma}
\begin{proof}
It suffices to show that $\{y_1,y_i\} \in y_1 A$ for all $i>1$. We can write $\{y_1,y_i\} = \sum_{j=0}^m f_j y_1^j$ where each $f_j\in \kk[y_2,\hdots,y_n]$. We have,
\[
\tornado \sum_{j=0}^m f_j y_1^j
    = \{ \tornado y_1, y_j\}
    = \{ g(y_1), g(y_i) \} = g(\{y_1,y_j\}) 
    = g\left( \sum_{j=0}^m f_j y_1^j\right) 
    = \sum_{j=0}^m f_j g(y_1)^j
    = \tornado^j \sum_{j=0}^m f_j y_1^j.
\]
Comparing terms of degree zero in $y_1$, we have $f_0=\tornado f_0$. Since $\tornado \neq 1$, $f_0 = 0$ and the result follows.
\end{proof}

When $A$ is a quadratic Poisson algebra, hence $\NN$-graded, Lemma \ref{lem.normal} implies that $A$ has a degree one Poisson normal element. Moreover, when $\tornado \neq -1$, the following result holds.

\begin{lemma}
\label{lem.subalg}
Suppose that $A=\kk[x_1,\hdots,x_n]$ is a quadratic Poisson algebra and $g$ a Poisson reflection of $A$ of order not $2$.
\begin{enumerate}
\item There is a basis $\{y_1,\hdots,y_n\}$ of $A$ such that $g(y_1)=\tornado y_1$ and $g(y_i)=y_i$ for all $i>1$
with $\tornado \in \kk^\times$ a root of unity, $\tornado \neq \pm 1$.
\item $A$ is $\ZZ^2$-graded by setting $\deg y_1 = (1,0)$ and $\deg y_i = (0,1)$ for all $i>1$.
\item $A=C[y_1;\alpha]$ where $C$ is a quadratic Poisson algebra generated by the $\{y_i\}_{i>1}$,
and $\alpha$ is a Poisson derivation of $C$.
\item $A^\grp{g}$ is a quadratic Poisson algebra.
\end{enumerate}
\end{lemma}
\begin{proof}
(a) This follows from the definition of a Poisson reflection, given in Section \ref{sec.background}.

(b) Since $A$ is a quadratic Poisson algebra, we have
\[ \{ y_i, y_j \} = \sum_{k \leq \ell} c_{k\ell}^{ij} y_k y_\ell,\]
for scalars $c_{k\ell}^{ij} \in \kk$. Note that $c_{k\ell}^{ii}=0$ for all $k,\ell,i$. Suppose $i>j > 1$. Since $g$ is a Poisson automorphism and by part (a),
\begin{align*}
\sum_{k \leq \ell} c_{k\ell}^{ij} y_k y_\ell
	= \{ y_i, y_j \}
	= \{ g(y_i), g(y_j) \}
	= g\left( \sum_{k \leq \ell} c_{k\ell}^{ij} y_k y_\ell \right)
	= \tornado^2 c_{11}^{ij} y_1^2 + \tornado \left( \sum_{1 < \ell} c_{1\ell}^{ij} y_1 y_\ell\right) + \sum_{1<k \leq \ell} c_{k\ell}^{ij} y_k y_\ell,
\end{align*}
which implies that $c_{1\ell}^{ij}=0$ for all $\ell\geq1$. Thus, in this case, $\{y_i,y_j\} \in \Span_\kk\{ y_ky_\ell : k,\ell > 1\}$, so $\deg(\{y_i,y_j\}) = (0,2)$.  

Similarly, for $j>1$,
\begin{align*}
\tornado \sum_{k \leq \ell} c_{k\ell}^{1j} y_k y_\ell
	= \{ \tornado y_1, y_j \}
	= \{ g(y_1), g(y_j) \}
	= g\left( \sum_{k \leq \ell} c_{k\ell}^{1j} y_k y_\ell \right)
	= \tornado^2 c_{11}^{1j} y_1^2 +
	 \tornado\left( \sum_{1 < \ell} c_{1\ell}^{1j} y_1 y_\ell\right) + \sum_{1<k \leq \ell} c_{k\ell}^{1j} y_k y_\ell,
\end{align*}
which implies $c_{11}^{1j}=c_{k\ell}^{1j}=0$ for $1<k \le \ell$. That is, $\{y_1,y_j\} \in \Span_\kk\{y_1y_j : j > 1\}$, so $\deg(\{y_1,y_j\})=(1,1)$. Hence $A$ is $\ZZ^2$-graded.

(c) By part (b), $A$ is $\ZZ^2$-graded and Lemma \ref{lem.Z2graded} (c) implies the result. 

(d) This follows from Example \ref{ex.ore}.
\end{proof}

The next example, which is an extension of \cite[Example 5.4]{KKZ3} to the Poisson setting, shows that the conclusion of Lemma \ref{lem.subalg} (c) fails when the order of the Poisson automorphism is two.

\begin{example}
Let $\scH_1=\kk[x,y,z]$ be the homogenized first Weyl Poisson algebra with bracket $\{x,y\}=z^2$ and $\{z,-\}=0$. We define a Poisson reflection of $\scH_1$ by $g(x)=x$, $g(y)=y$, and $g(z)=-z$. Then $|g|=2$ and the fixed subring is $\scH_1^{\grp{g}} = \kk[x,y,z^2]$. We claim that $\scH_1 \niso C[z;\alpha]$ for any Poisson subalgebra $C$ of $\scH_1$ generated by two elements of degree one. Suppose $\scH_1 = C[z;\alpha]$. Observe that $z$ is the only Poisson normal element of degree one in $\scH_1$. Thus $C=\scH_1/(z)$, which has zero Poisson bracket. Since $z$ lies in the Poisson center of $\scH_1$, then $\scH_1$ must have zero Poisson bracket, a contradiction. Thus, $\scH_1 \niso C[z;\alpha]$. Moreover, $\scH_1^{\grp{g}} \niso \scH_1$, as we show in Theorem \ref{thm.hweyl}.
\end{example}

\subsection{The two variable case}

We end this section with a straightforward analysis of fixed rings of quadratic Poisson algebra structures on $\kk[x,y]$. This will be useful in illustrating a more general setting later.

\begin{lemma}
\label{lem.iso1}
Let $A=\kk[x,y]$ be a quadratic Poisson algebra, so $\{x,y\}=f \in A_2$.
Then, up to isomorphism, $f=pxy$ for some $p \in \kk$ or $f=x^2$.
Moreover, these Poisson structures are nonisomorphic up to replacement of $p$ by $-p$.
\end{lemma}
\begin{proof}
By \cite[Theorem 4.2]{GXW}, we may assume that all Poisson isomorphisms preserve the (standard) grading on $A$.

If $f=0$, then there is nothing to prove. Otherwise, factor $f$ as $(\mu_{11}x+\mu_{12}y)(\mu_{21}x+\mu_{22}y)$ for some $\mu_{ij} \in \kk$. Suppose $q=\mu_{11}\mu_{22}-\mu_{12}\mu_{21} \neq 0$. Set $X=\mu_{11}x+\mu_{12}y$ and $Y=\mu_{21}x+\mu_{22}y$ so that $\kk[X,Y]=\kk[x,y]$. Moreover,
\[
\{X,Y\} = \{\mu_{11}x+\mu_{12}y,\mu_{21}x+\mu_{22}y\} = q\{x,y\} = qf = q(\mu_{11}x+\mu_{12}y)(\mu_{21}x+\mu_{22}y)
= qXY.
\]


Now, suppose $q=0$. Then after a possible linear transformation we may assume that $f=(\mu_{11}x+\mu_{12}y)^2$ with $\mu_{11}\neq 0$.
In this case set $X=\mu_{11}x+\mu_{12}y$ and $Y=\mu_{11}\inv y$. Then $\kk[X,Y]=\kk[x,y]$ and
\[
\{X,Y\}=\{\mu_{11}x+\mu_{12}y,\mu_{11}\inv y\}
= \{x,y\} = f = X^2.
\]


The last claim can be similarly proved.
\end{proof}

For a group $G$, we denote by the $\ex(G)$ the exponent of $G$. The next result is a rigidity result for quadratic Poisson algebras.

\begin{proposition}
\label{prop.n2}
Let $A=\kk[x,y]$ be a quadratic Poisson algebra with nonvanishing Poisson bracket. Suppose $G$ is a finite subgroup of $\GrAutp(A)$. If $A^G\cong A$ as Poisson algebras, then $G$ is trivial. 
\end{proposition}
\begin{proof}
By the STC theorem (Theorem \ref{thm.stc}), $G$ is generated by Poisson reflections. Moreover by Lemma \ref{lem.iso1}, we only need to deal with the following two cases. 

Case 1: $\{x,y\}=pxy$ for some $p\in \kk^\times$. By Example \ref{ex:normalskew}, the degree one Poisson normal elements of $A$ are scalar multiples of $x$ or $y$. Let $g\in \GrAutp(A)$. Since $g$ preserves Poisson normal elements within the same degree, it is easy to see that $g(x)=\mu x$ and $g(y)=\nu y$ for some $\mu,\nu\in k^\times$ by using the fact that $g\{x,y\}=\{g(x),g(y)\}$. Now, assume that $g$ is a Poisson reflection. By definition, we can either write $g(x)=\tornado x$ and $g(y)=y$  or  $g(x)=x$ and $g(y)=\tornado y$ for some $m$th root of unity $\tornado\neq 1$. Since $G$ is generated by such Poisson reflections, there is a group decomposition $G=G_x\times G_y$ where the subgroups $G_x$ and $G_y$ act on $\kk[x]$ and $\kk[y]$ separately. So $A^G=\kk[x^m,y^n]$ with Poisson bracket $\{x^m,y^n\}=(pmn)x^my^n$ where $m=\ex(G_x)$ and $n=\ex(G_y)$. Now one can immediately check that $A^G\cong A$ if and only if $mn=1$ if and only if $G$ is trivial.



Case 2: $\{x,y\}=x^2$. The only degree one Poisson normal elements of $A$ are scalar multiples of $x$. So for any $g\in \GrAutp(A)$, we can write  $g(x)=\mu_{11} x$ and $g(y)=\mu_{21}x+\mu_{22}y$ for $\mu_{ij} \in \kk$ with $\mu_{11}\mu_{22}\neq 0$. Then,
\[ (\mu_{11}\mu_{22})x^2 
= \{\mu_{11}x,\mu_{21}x+\mu_{22}y\} 
= \{g(x),g(y)\} = g(\{x,y\}) 
= (\mu_{11}x)^2.\]
Thus, $\mu_{11}=\mu_{22}$ and $g$ has a single eigenvalue. It follows that $A$ has no Poisson reflections. Hence $G$ must be trivial.  
\end{proof}

\section{Shephard-Todd-Chevalley theorem for skew-symmetric Poisson algebras}
\label{sec.skew}

A Poisson algebra $A=\kk[x_1,\hdots,x_n]$ is \emph{skew-symmetric} (or \emph{has skew-symmetric Poisson structure}) if $\{x_i,x_j\}=q_{ij} x_ix_j$ for all $1 \leq i,j \leq n$ where $(q_{ij}) \in M_n(\kk)$ is skew-symmetric. Throughout this section, let $A$ be a skew-symmetric Poisson algebra.

Suppose $g$ is a Poisson reflection of $A$. By the STC theorem (Theorem \ref{thm.stc}), there exists a basis $\{y_1,\hdots,y_n\}$ of $A_1$ such that $g(y_1)=\tornado y_1$ for some $m$th root of unity $\tornado\neq 1$ and $g(y_i)=y_i$ for all $i > 1$. If the bracket on the $y_i$ are again of skew-symmetric form above, then $A^{\grp{g}}=\kk[y_1^m,y_2,\hdots,y_n]$ and clearly this bracket is of skew-symmetric form. Hence, we aim to understand when the bracket on the eigenbasis must have skew-symmetric form.

Many of our subsequent results are Poisson analogues of results for quantum affine spaces $\kk_\bp[x_1,\hdots,x_n]$ (see \cite{FMdg,KKZ3,KKZ1}).

\begin{lemma}
Suppose that $A$ is a skew-symmetric Poisson algebra with $q_{ij} \neq 0$ for all $i\neq j$.
Let $\phi$ be a graded Poisson automorphism of $A$.
Then $\phi$ is given by a monomial matrix.
\end{lemma}
\begin{proof}
By Example \ref{ex:normalskew}, the only Poisson normal elements in $A_1$ are $x_1,\ldots,x_n$. So any graded Poisson automorphism of $A$, when restricted to  $A_1$, must be a permutation of $x_1,\ldots,x_n$ up to scalar multiplication. It follows that each graded automorphism of $A$ may be represented (by its action on $A_1$) as a monomial matrix.
\end{proof}

Let $A=\kk[x_1,\hdots,x_n]$ be a skew-symmetric Poisson algebra. A $\kk$-basis of $A$ consists of the set of all (ordered) monomials in these variables. Let $\bI = (i_1,\hdots,i_n) \in \NN^n$ be a multiindex and denote the monomial $x_1^{i_1} \cdots x_n^{i_n}$ by $x^\bI$.
For $f \in A$, we denote by $\supp(f)$ the set of multiindices $\bI$ such that the coefficient of $x^\bI$ in $f$ is nonzero. 

In Section \ref{sec.PoissonNormal}, we introduced the set of all Poisson derivations of $A$, denoted by $\pder(A)$. By Lemma \ref{lem.pder}, each Poisson normal element $x_i$ generates a Poisson derivation $\pi_{x_i}=\{x_i,-\}/x_i$ for $1\le i\le n$. These derivations form an abelian subgroup of $\pder(A)$ under addition, which we denote by $\cL$. Moreover, each monomial $x^\bI$ is Poisson normal for all $\bI\in \NN^n$. We denote the Poisson derivation corresponding to $x^\bI$ by $\alpha_\bI$. It is clear that $\alpha_\bI\in \mathcal L$.

For any $\alpha \in \cL$, we define 
\[A_\alpha = \{ f \in A : \{f,a\} = \alpha(a)f \text{ for all } a \in A\}\] 
as a $\kk$-vector subspace of $A$. We say an element $f \in A_\alpha$ is \emph{$\cL$-homogeneous of degree $\alpha$}.

We now adapt to the Poisson case a result of Ferraro and Moore \cite[Lemma 3.5]{FMdg}, which itself is a generalization of a result of Kirkman, Kuzmanovich, and Zhang \cite[Lemma 3.5]{KKZ1}.

\begin{lemma}\label{lem.hom}
Let $f\in A_\alpha$ for some $\alpha\in \mathcal L$. Then $x^\bI\in A_\alpha$ for any $\bI\in \supp(f)$.
\end{lemma}
\begin{proof}
By definition, $f \in A_\alpha$ is Poisson normal with associated derivation $\alpha$. Write $f = \sum_{\bI \in \supp(f)} c_{\bI} x^{\bI}$. It suffices to show that $\alpha_\bI(x_j)=\alpha_{\bI'}(x_j)$ for any two $\bI,\bI'\in \supp(f)$ and all $j$ since the $\cL$-degree of each monomial in the support of $f$ will be the same. The case is trivial when $|\supp(f)|=1$, so assume $|\supp(f)|\geq 2$.

Suppose $x^\bK$ divides $x^\bI$ for all $\bI \in \supp(f)$. That is, we may write $f=x^\bK g$ with some $g \in A$. Now for all $x_j$,
\[ \alpha(x_j)f = \{ f, x_j \} = \{ x^\bK g, x_j \} = \{ x^\bK, x_j \}g + \{ g, x_j \} x^{\bK}
	= \alpha_{\bK}(x_j) f + \{ g, x_j \} x^{\bK}.\]
Thus, $(\alpha - \alpha_{\bK} )(x_j) g x^{\bK} = \{ g, x_j \} x^{\bK}$. As $A$ is a domain
it follows that $\{ g, x_j \} = (\alpha - \alpha_{\bK} )(x_j) g$. That is, $g$ is Poisson normal.
If the result holds for $g$, then it holds for $f$. Therefore, we may reduce to the case that there is no such $x^\bK$.

Let $\overline{(~)}$ denote the image of an element in $\overline{A}=A/(x_j)$. Choose $j\in \{1,\hdots,n\}$.
By Poisson normality of $f$ and homogeneity of the Poisson bracket, we have coefficients $b_i\in \kk$ such that 
\[ \{ f, x_j \} = \left( \sum_i b_i x_i \right) f\]
and so
\[ 0 = \left( \sum_{i \neq j} b_i \overline{x}_i \right) \overline{f} \in \overline{A}.\] 
By assumption, $\overline{f}\neq 0$. Thus, $b_i=0$ for all $i \neq j$ and $\{ f, x_j \} = \beta_j x_j f$.
It follows that $\alpha_\bI(x_j) = \alpha_{\bI'}(x_j) = \beta_j x_j$ 
for all $\bI,\bI' \in \supp(f)$. This proves our result. 
\end{proof}

As a consequence of Lemma \ref{lem.hom}, we have $A=\bigoplus_{\alpha\in \cL} A_\alpha$. We now show that this decomposition respects the Poisson structure on $A$.

\begin{lemma}
\label{lem.decomp}
For all $\alpha,\beta \in \cL$, $A_\alpha \cdot A_\beta \subseteq A_{\alpha+\beta}$ and $\{A_\alpha, A_\beta\} \subseteq A_{\alpha+\beta}$. Thus, $A$ is a $\cL$-graded Poisson algebra.
\end{lemma}
\begin{proof}
Let $f \in A_\alpha$ and $g \in A_\beta$. For any $a \in A$, we have
\[ \{ fg,a \}
    = \{ f,a \}g + \{g,a\}f
    = \alpha(a)fg + \beta(a)fg
    = (\alpha+\beta)(a)fg.\]
Moreover, since $\alpha$ and $\beta$ commute, we have
\begin{align*}
\{ \{f,g\}, a\} &= - \{ a, \{f,g\} \}
= \{ f, \{ g,a\} \} + \{ g, \{a,f\} \} \\
&= \{ f, \beta(a)g \} - \{g, \alpha(a)f\} \\
&= \left( \{ f,\beta(a) \}g + \{f,g\}\beta(a) \right)
- \left( \{ g,\alpha(a) \}f + \{g,f\}\alpha(a) \right) \\
&= \left( \alpha(\beta(a))-\beta(\alpha(a) \right) fg + (\alpha(a)+\beta(a))\{f,g\} \\
&= (\alpha+\beta)(a)\{f,g\}.\qedhere
\end{align*}
\end{proof}

By Lemma \ref{lem.hom}, we can take any monomials $f\in A_\alpha$ and $g\in A_\beta$ such that $\{f, g\} = q_{\alpha,\beta}fg$ for some $q_{\alpha,\beta} \in \kk$. We define a map $\chi: \cL \times \cL \to \kk$ by setting $\chi(\alpha,\beta)=q_{\alpha,\beta}$. In case $\alpha= x_i$ and $\beta= x_j$, then $\chi(x_i,x_j)=q_{ij}$.

\begin{lemma}
\label{lem.ghomog}
The bilinear map $\chi: \cL \times \cL \to \kk$ is a well-defined skew-symmetric additive bicharacter. Moreover, $\{ f,g \} = \chi(\alpha,\beta)fg$ for any $f \in A_\alpha$ and $g \in A_\beta$. 
\end{lemma}
\begin{proof}
The value of $\chi(\alpha,\beta)$ does not depend on the choice of monomials $f \in A_\alpha$ and $g \in A_\beta$ because of Lemma \ref{lem.hom}. Additivitity of $\chi$ is clear. Moreover,
\[ \chi(\alpha,\beta)fg = \{f,g\} = - \{g,f\} = -\chi(\beta,\alpha)fg.\]
Thus, $\chi$ is skew-symmetric.
\end{proof}

At this point we are able to determine how Poisson normal elements interact in a skew-symmetric Poisson algebra.

\begin{lemma}
\label{lem.pnormal}
Let $A=\kk[x_1,\hdots,x_n]$ be a skew-symmetric Poisson algebra.
Let $f,g \in A$ be $\mathbb N$-homogeneous.
\begin{enumerate}
\item The element $f$ is Poisson normal with associated derivation $\alpha$
if and only if $f$ is $\cL$-homogeneous of degree $\alpha$.
\item If $f, g \in A$ are Poisson normal elements of $A$, then there exists $p \in \kk$ such that $\{f,g\} = pgf$ for some $p \in \kk$.
\item If $f \in A$ is Poisson normal, then $\chi(\alpha_\bI,\alpha_\bJ)=0$ for all $\bI,\bJ \in \supp(f)$.
\item If $\{f,g\} = qfg$ for some $q \in \kk^\times$, then $\supp(f) \cap \supp(g) = \emptyset$.
\item If $\{f,g\}=0$, then $\chi(\alpha_\bI,\alpha_\bJ)=0$ for all $\bI \in \supp(f)$ and $\bJ \in \supp(g)$ such that $\bI,\bJ \notin\supp(f) \cap \supp(g)$.
\end{enumerate}
\end{lemma}
\begin{proof}
For (a), it follows from Lemma \ref{lem.hom}.

For (b), we note that $f$ and $g$ are $\cL$-homogeneous by (a) and so the result follows by Lemma \ref{lem.ghomog}.

For (c), we observe that $\chi(\alpha,\alpha)=0$ for all $\alpha \in \cL$.

Finally for (d) and (e), suppose $\bI \in \supp(f) \cap \supp(g)$.
Now $\supp\left((x^{\bI})^2\right) \in \supp(fg)$ but on the other hand,
\[ 
qfg = \{ f,g \} 
	= \sum_{\substack{\bI \in \supp(f) \\ \bJ \in \supp(g)}} \{ x^\bI, x^\bJ \}
	= \sum_{\substack{\bI \in \supp(f) \\ \bJ \in \supp(g)}} \chi(\alpha_\bI, \alpha_\bJ)x^\bI x^\bJ.
\]
Hence, $\supp\left((x^{\bI})^2\right) \notin \supp(\{ f,g \})$, a contradiction.
\end{proof}

\begin{definition}
Let $A=\kk[x_1,\hdots,x_n]$ be a skew-symmetric Poisson algebra such that $\{x_i,x_j\}=q_{ij} x_ix_j$ for all $1 \leq i,j \leq n$ where $(q_{ij}) \in M_n(\kk)$ is skew-symmetric. Set $[n]=\{1,2,\hdots,n\}$.
\begin{enumerate}
    \item 
For $i \in [n]$, we define the \emph{block} of $i$ to be
\[ B(i) = \{ i' \in [n] : q_{ik}=q_{i'k} \text{ for all } k \in [n] \}.\]
\item We use the blocks to define an equivalence relation on $[n]$ defined by setting $i \sim j$ if $B(i)=B(j)$.
\item This equivalence relation then partitions $[n]$ into disjoint blocks as
\[ [n] = \bigcup_{i \in W} B(i)\]
where $W$ is an index set. We call this the \emph{block decomposition of $[n]$}
and denote $B(i)$ by $B_w$ for $w \in W$. 
\end{enumerate}
\end{definition}

Essentially, $i' \in B(i)$ if row $i$ and row $i'$ of $(q_{ij})$ are identical. One can also check that $q_{ii'}=0$ for all $i,i' \in B_w$ for any $w\in W$. We further denote $P_w = \bigoplus_{i \in B_w} \kk x_i$ for each $w$. For a given $B_w$, let $\GrAutpw(A)$ be the subgroup of $\GrAutp(A)$ consisting of those automorphisms $\theta$ satisfying
\begin{enumerate}
\item $\theta(x_s) = x_s$ for all $s \notin B_w$, and
\item $\theta(x_s) \in P_w$ for all $s \in B_w$.
\end{enumerate}
If $G$ is a subgroup of $\GrAutp(A)$, then we denote by $G_w$ the intersection $G \cap \GrAutpw(A)$.

\begin{lemma}
\label{lem.graut}
Let $A=\kk[x_1,\hdots,x_n]$ be a skew-symmetric Poisson algebra and let $\theta \in \GrAutp(A)$ be a Poisson reflection.
Then $\theta \in \GrAutpw (A)$ for some $w \in W$.
\end{lemma}
\begin{proof}
Let $v$ be a non-invariant eigenvector of $\theta$ such that $\theta(v)=\tornado v$ for some root of unity $\tornado \neq 1$ and $\theta$ is the identity on $A/(v)$. By Lemma \ref{lem.normal}, $v$ is Poisson normal in $A_1$. If $i,i'\in \supp(v)$, then $\{x_i,-\}=\{x_{i'},-\}=\{v,-\}$ by Lemma \ref{lem.hom}. So $\supp(v) \subset B_w$ for some $w \in W$ and hence $v\in P_w$. 

As $\theta$ is the identity on $A/(v)$, then for all $i$ we have $\theta(x_i)=x_i+c_i v$ for some $c_i \in \kk$. Since $v \in P_w$, it is immediate that $\theta(x_s) \in P_w$ for all $s \in B_w$.
Now choose $s \notin B_w$. The Poisson normality of $x_s$ implies Poisson normality of $\theta(x_s)$. By Lemma \ref{lem.hom}, $\supp(\theta(x_s))\subset B_{w'}$ for some $w'\in W$. But $\theta(x_s)=x_s+c_s v$. This implies that $c_s=0$ and $\theta(x_s) = x_s$ for all $s \notin B_w$. Then our result follows from the definition of $\GrAutpw (A)$.
\end{proof}

We may now prove the main result of this section, which is a version of the STC theorem for skew-symmetric Poisson algebras.

\begin{theorem}
\label{thm.fixed}
Let $A=\kk[x_1,\hdots,x_n]$ be a skew-symmetric Poisson algebra and let $G$ be a finite subgroup of $\GrAutp(A)$. Then $A^G$ has skew-symmetric Poisson structure if and only if $G$ is generated by Poisson reflections.
\end{theorem}
\begin{proof}
One direction is clear by the STC theorem (Theorem \ref{thm.stc}). Now assume that $G$ is generated by Poisson reflections. By Lemma \ref{lem.graut}, each generator of $G$ belongs to some $G_w$. Moreover, if $B_w \cap B_{w'}=\emptyset$, then generators of $G_w$ commute with those of $G_{w'}$. Consequently, $G=\prod_{w \in W} G_w$ and each $G_w$ is generated by Poisson reflections.

Fix some $w \in W$ and suppose, up to reordering, that $\{x_1,\hdots,x_k\}$ is a basis of $P_w$. Hence, $x_1,\hdots,x_k$ generate a Poisson subalgebra of $A$, which we denote by $A_w$. Furthermore, $G_w$ acts trivially on the remaining generators of $A$ and so $A^{G_w}=A_w^{G_w}[x_{k+1},\hdots,x_n]$. We claim that $A^{G_w}$ has skew-symmetric structure.

Since $G_w$ is a reflection group of $A_w$, then $A_w^{G_w}=\kk[u_1,\hdots,u_k]$ for some homogeneous $u_i$ in the $x_1,\hdots,x_k$. Since $q_{\ell\ell'}=q_{\ell\ell}=0$ for all $\ell,\ell' \in B_\omega$, then $A_w$ has zero Poisson bracket. Hence $\{u_i,u_j\}=0$ for all $1 \leq i,j \leq k$. Thus, $A_w^{G_w}$ is trivially a skew-symmetric Poisson algebra. Moreover, the $x_{k+1},\hdots,x_n$ also generate a Poisson subalgebra of $A^{G_w}$ with skew-symmetric Poisson structure. Finally, recall that if $1 \leq i \leq k$ and $j>k$, then $q_{ij}=q_{1j}$. Thus, since $u_i$ is homogeneous, say of degree $d_i$, then $\{u_i,x_j\} = q_{1j}^{d_i} u_ix_j$. This proves the claim.

Let $w' \in W$ with $B_w \neq B_{w'}$. Then $G_{w'}$ commutes with $G_w$ and hence $G_{w'}$ acts naturally on the fixed Poisson algebra $A^{G_w}$. Thus, the result follows by induction.
\end{proof}

\section{Rigidity of some quadratic Poisson algebras}
\label{sec.rigidity}

In contrast to the skew-symmetric Poisson algebras of the previous section, many quadratic Poisson structures are (graded) rigid. That is, the Poisson structure is not preserved under taking fixed subrings. In some cases, as we show below, we prove a much stronger result - there do not exist any Poisson reflections.

\subsection{Jacobian brackets}
\label{sec.jacobian}
Let $A=\kk[x,y,z]$ be a Poisson algebra with Jacobian bracket determined by the potential
\[ f_{p,q} := \frac{p}{3}(x^3+y^3+z^3)+qxyz, \quad p,q \in \kk.\]
In the generic case, we will show that $A$ has no Poisson reflections. However, under certain cases, our analysis mirrors that of Section \ref{sec.skew}.
We begin by considering two special cases.

\begin{lemma}
\label{lem.jac1}
Let $A=\kk[x,y,z]$ be a Poisson algebra with Jacobian bracket determined by the potential $f_{p,q}$ with $p \neq 0$ and $q=0$. Then $A$ has no Poisson reflections.
\end{lemma}
\begin{proof}
The Poisson bracket in this case is given by
\[ \{x,y\}=pz^2, \quad \{y,z\}=px^2, \quad \{z,x\}=py^2.\]
Let $u=\mu x+\nu y+\eta z$ be a degree one Poisson normal element of $A$, with $\mu,\nu,\eta \in \kk$. A computation shows that $\{x,u\} = p(\nu z^2-\eta y^2)$. By Poisson normality, $\{x,u\} \in (u)$, which implies that $\mu=0$. A similar computation shows that $\nu=\eta=0$. That is, $A$ has no degree one Poisson normal elements. Hence, by Lemma \ref{lem.normal}, $A$ has no Poisson reflections.
\end{proof}

The potential in Lemma \ref{lem.jac1} has isolated singularities, that is, the graded quotient algebra $A/f$ is not regular but its graded localization $(A/f)_{(\mathfrak p)}$ is regular for any homogeneous prime ideal $\mathfrak p$ of $A/f$. The next lemma gives an example where the potential does not have isolated singularities.

\begin{lemma}
\label{lem.jac2}
Let $A=\kk[x,y,z]$ be a Poisson algebra with Jacobian bracket determined by the potential $f_{p,q}$ with $p=0$ and $q \neq 0$. If $g$ is a Poisson reflection of $A$, then, up to isomorphism, $A^{\grp{g}}=\kk[x^n,y,z]$ with $n={\rm ord}(g)>1$ and Poisson bracket
\[ \{x^n,y \} = nqx^n y, \quad \{y,z\}=qyz, \quad \{z,x^n\} = nqx^nz.\]
Moreover, the fixed Poisson subring $A^{\grp{g}}$ does not have Jacobian structure. 
\end{lemma}
\begin{proof}
The Poisson bracket in this case is given by
\[ \{x,y\} = qxy, \quad \{y,z\} = qyz, \quad \{z,x\} = qxz.\]
Example \ref{ex:normalskew} and our assumption show that the degree one normal elements are exactly $\mu x,\nu y,\eta z$ for $\mu,\nu,\eta\in \kk^\times$. If $g$ is a Poisson reflection, then $g$ fixes one of these up to scalar.

Without loss of generality, we may take $g(x)=\tornado x$ for some primitive $n$th root of unity $\tornado$ with $n>1$. Since $g$ is a graded Poisson automorphism, then $g$ must map degree one Poisson normal elements to degree one Poisson normal elements. Thus, we have one of two cases below:
\begin{itemize}
    \item[(i)] $g(y)=\mu y$ and $g(z)=\eta z$, or 
    \item[(ii)] $g(y)=\eta z$ and $g(z)=\mu y$, 
\end{itemize}
for some $\mu,\eta \in \kk^\times$. In the second case, we have
\[ q\mu\eta zy = qg(y)g(z) = g(qyz) = g(\{y,z\}) = \{ g(y),g(z)\}
    = \{ \eta z, \mu y \} = \mu\eta (-qyz).\]
Since $q \neq 0$ and $\ch\kk \neq 2$, we have a contradiction. Thus, we conclude that $g(y)=y$ and $g(z)=z$. The fixed subring in this case is then $\kk[x^n,y,z]$ with Poisson bracket given in the statement. By Equation \eqref{equ.modular}, the modular derivation $\phi_\eta$ of $A^{\langle g\rangle}$ is given by $\phi_\eta(x^n)=0$, $\phi_\eta(y)=q(1-n)y$ and $\phi_\eta(z)=q(n-1)z$. Hence $A^{\langle g\rangle}$ is not unimodular and so it does not have the Jacobian structure by Example \ref{ex.unimodular} (b).
\end{proof}

In the remaining, we will consider the cases of the potential $f_{p,q}$ with $pq\neq 0$.

\begin{lemma}
\label{lem.pq}
Let $A=\kk[x,y,z]$ be a Poisson algebra with Jacobian bracket determined by the potential $f_{p,q}$ with $pq \neq 0$. Then $A$ has a degree one Poisson normal element if and only if $p=-\omega q$ for some $\omega \in \kk^\times$ with $\omega^3=1$.
\begin{enumerate}
\item If $\omega=1$, then the degree one Poisson normal elements are, up to scalar multiple, $x+\gamma y + \gamma^2 z$ where $\gamma \in \kk^\times$ satisfies $\gamma^3=1$. 
\item If $\omega$ is a primitive third root of unity, then the degree one Poisson normal elements of $A$ are, up to scalar multiple,
\[ 
\omega^2x+ y + z, \quad
x+y+\omega^2 z, \quad
x+\omega^2 y + z.
\]
\end{enumerate}
\end{lemma}
\begin{proof}
Let $u=\mu x + \nu y + \eta z$ be a nonzero degree one normal element of $A$. If $\mu = 0$, then 
\[ \{y,u\} = \eta(px^2+qyz).\]
Since $px^2+qyz$ is an irreducible quadratic element in $A$, then $\eta=0$. A similar argument shows that $\nu=0$. Thus, by symmetry, we have $\mu,\nu,\eta \neq 0$. We assume, henceforth, that all three are nonzero. Now, 
\begin{align}
\label{eq.xu}
\{x,u\} = p(\nu z^2 - \eta y^2) + qx(\nu y - \eta z) \in (u).
\end{align}
Let $v=\mu' x + \nu' y + \eta' z \in A_1$ such that $\{x,u\}=uv$. Since $x^2$ is not a summand of $\{x,u\}$ and $\mu \neq 0$, then $\mu'=0$. Thus,
\begin{align}
\label{eq.uv}
uv = (\nu\nu' y^2 + \eta\eta' z^2) + \mu\nu' xy + \mu\eta' xz + (\nu\eta' + \eta\nu') yz.
\end{align}
Hence,
\begin{align}
\label{eq.nueta1} 
\nu\nu' = -p\eta, \qquad \eta\eta' = p\nu.
\end{align}
Therefore $\nu'=-p\eta/\nu$ and $\eta' = p\nu/\eta$, which implies that $\nu',\eta' \neq 0$. Furthermore,
\[ 0 = \nu\eta' + \eta\nu' = \nu(p\nu/\eta) + \eta(-p\eta/\nu) =\frac{p}{\nu\eta}(\nu^3 - \eta^3).\]
Now by symmetry we have $\mu^3=\nu^3=\eta^3$. By rescaling, we may assume henceforth that $\mu=1$. Thus, $\nu$ and $\eta$ are third roots of unity. Returning to \eqref{eq.xu} and \eqref{eq.uv}, the coefficients of $xy$ and $xz$ gives
\begin{align}
\label{eq.nueta2} 
\nu'= q\nu, \eta'=-q\eta.
\end{align}
Our analysis now splits into several cases.

\noindent\textbf{Case 1:} Suppose $\nu=\eta=\omega$ for some $\omega \in \kk^\times$ with $\omega^3=1$ (note that we may have $\omega=1$). By \eqref{eq.nueta1} and \eqref{eq.nueta2}, $\nu' = -p$ and $\nu'= q\omega$. Hence, $p=-\omega q$. 
Conversely, we can check that under the condition $p=-\omega q$, $\{y,u\} \in (u)$ and $\{z,u\} \in (u)$. Thus, $u$ is Poisson normal.

\noindent\textbf{Case 2:} Suppose $\nu=1$ and $\eta=\omega$ for some $\omega$ with $\omega^3=1$.
Note if $\omega=1$, then this reduces to Case 1.
Again using \eqref{eq.nueta1} and \eqref{eq.nueta2}, we have $\nu' = -\omega p$ and $\nu'= q$. Hence, $p=-\omega^2 q$. 

\noindent\textbf{Case 3:} Suppose $\eta=1$ and $\nu=\omega$ for some $\omega$ with $\omega^3=1$.
Since this is identical to Case 2, we have $p=-\omega^2 q$.

\noindent\textbf{Case 4:} Suppose $\nu=\omega$ and $\eta=\omega^2$ for some $\omega$ with $\omega^3=1$.
Note if $\omega=1$, then this reduces to Case 1.
Using \eqref{eq.nueta1} and \eqref{eq.nueta2} again, we have $\omega \nu' = -\omega^2 p$ and $\nu'= q\omega$. Hence, $q=- p$.

Conversely, one may check, in each case, that under the given condition, $\{y,u\} \in (u)$ and $\{z,u\} \in (u)$. Thus, $u$ is Poisson normal in each case.
\end{proof}

We show in the next result that we may diagonalize the two cases of Lemma \ref{lem.pq} into the case of Lemma \ref{lem.jac2}

\begin{theorem}
\label{thm.jac}
Let $A=\kk[x,y,z]$ be a Poisson algebra with Jacobian bracket determined by some potential $f_{p,q}\neq 0$. If $A$ has a Poisson reflection $g$, then $A$ is Poisson isomorphic to the Poisson algebra on $\kk[x,y,z]$ with Jacobian bracket determined by $f_{p,q}$ with $p=0$.
\end{theorem}
\begin{proof}
Observe that the cases when $q=0$ or $p=0$ are addressed in Lemmas \ref{lem.jac1} and \ref{lem.jac2}, respectively. It remains to show that the Poisson algebras appearing in the two cases of Lemma \ref{lem.pq} are isomorphic to one appearing in Lemma \ref{lem.jac2}.

Consider the case $p=-q$. That is, $A$ has Poisson bracket
\[
\{x,y\} = q(xy-z^2), \quad
\{y,z\} = q(yz-x^2), \quad
\{z,x\} = q(xz-y^2).
\]
Let $\gamma$ be a primitive third root of unity. 
The degree one Poisson normal elements, up to scalar multiple, are
\[
u = x + y + z, \quad
v = x + \gamma y + \gamma^2 z, \quad
w = x + \gamma^2 y + \gamma z.
\]
This is a Poisson normal basis for $A_1$.
Set $\rho = \gamma q(1-\gamma)$. A computation shows that 
\[
\{ u,v \} = \rho uv, \quad
\{ v,w \} = \rho vw, \quad
\{ w,u\} = \rho uw.
\]
Similarly, in the case $p=-\omega q$ for $\omega$ a primitive third root of unity, the bracket on $A$ is given by
\[
\{x,y\} = q(xy-\omega z^2), \quad
\{y,z\} = q(yz-\omega x^2), \quad
\{z,x\} = q(xz-\omega y^2).
\]
The degree one Poisson normal elements, up to scalar multiple, are
\[ 
u = x+\omega y + \omega z, \quad
v = x+y+\omega^2 z, \quad
w = x+\omega^2 y + z.
\]
This is a Poisson normal basis for $A_1$.
Set $\rho = \omega q(\omega-1)$. A computation shows that 
$$\{ u,v \} = \rho uv, \quad \{ v,w \} = \rho vw, \quad
\{ w,u\} = \rho uw.
$$
In both cases, we may now refer to Lemma \ref{lem.jac2} for a description of Poisson reflections and the corresponding fixed subrings.
\end{proof}

Now we are ready to prove our rigidity theorem for $A$. 

\begin{theorem}\label{thm.rigidJac}
Let $A=\kk[x,y,z]$ be a Poisson algebra with Jacobian bracket determined by some potential $f_{p,q}\neq 0$. Suppose $G$ is a finite subgroup of $\GrAutp(A)$. If $A^G\cong A$ as Poisson algebras, then $G$ is trivial. 
\end{theorem}
\begin{proof}
Since $A^G\cong A$ as Poisson algebras, then $G$ is generated by Poisson reflections by the STC theorem (Theorem \ref{thm.stc}). By Theorem \ref{thm.jac}, we may take $p=0$ in the potential $f_{p,q}$ . Hence, $A$ becomes a skew-symmetric Poisson algebra with the skew-symmetric matrix 
$\left(\begin{smallmatrix}
0 & -q & q\\
q & 0 & -q\\
-q & q & 0
\end{smallmatrix}\right)$. As in the proof of Theorem \ref{thm.fixed}, we have a decomposition $G=G_x\times G_y\times G_z$ where the subgroups $G_x$, $G_y$, $G_z$ act on $\kk[x]$, $\kk[y]$, $\kk[z]$ separately. As a consequence, we have the fixed subring $A^G=\kk[x^m,y^n,z^\ell]$ with Poisson bracket given by 
\[
\{x,y\}=(qmn)xy,\quad \{y,z\}=(qn\ell)yz,\quad \{z,x\}=(qm\ell)zx
\]
where $n={\rm exp}(G_x),m={\rm exp}(G_y),\ell={\rm exp}(G_z)$. 
Since $A^G\cong A$, we obtain $qmn=qn\ell=qm\ell=q$ and $m=n=\ell=1$ by \cite[Theorem 4.6]{GXW}. Thus, $G$ is trivial. 
\end{proof}

\subsection{Coordinate ring of $n \times n$ matrices}
\label{sec.mat}
Let $M_n$ denote the ring of $n \times n$ matrices for some $n\ge 2$.
Per \cite[\S 3]{vancliff}, the Poisson bracket on the polynomial ring $\cO(M_n)=\kk[x_{ij}]_{1\le i,j\le n}$ is given by
\[
\{x_{im},x_{j\ell} \} = 0, \quad \{ x_{i\ell},x_{im} \} = x_{i\ell}x_{im}, \quad
 \{x_{i\ell},x_{j\ell}\} = x_{i\ell}x_{j\ell}, \quad \{x_{i\ell},x_{jm} \} = 2x_{im}x_{j\ell}
\]
with $i < j$ and $\ell < m$. This Poisson bracket can be realized as the semiclassical limit of the family of $n \times n$ quantum matrices $\{\cO_q(M_n)\}$ for $q \in \kk^\times$. In the case $n=2$, we simplify the presentation by setting $a=x_{11}$, $b=x_{12}$, $c=x_{21}$, and $d=x_{22}$. Then the above implies that the bracket on $\cO(M_2)=\kk[a,b,c,d]$ is
\begin{align*}
&\{a,b\} = ab	&	&\{a,c\}=ac	&	&\{a,d\}=2bc \\
&\{b,d\}=bd	&	&\{c,d\}=cd 	& 	&\{b,c\}=0.
\end{align*}


We aim to understand the Poisson reflections of $\cO(M_n)$. The next lemma is a Poisson analogue of \cite[Theorem 3.4]{Giso2}. Here we use $\supp(f)$ to denote the monomials that appear as nonzero summands of an element $f \in \cO(M_n)$.

\begin{lemma}
\label{lem.mnormal}
Let $u$ be a degree one Poisson normal element of $\cO(M_n)$. Then 
$u = \mu x_{1n} + \mu' x_{n1}$ for some $\mu,\mu' \in \kk$.
\end{lemma}
\begin{proof}
Let $u$ be a degree one normal element of $\cO(M_n)$. Suppose, by way of contradiction, that there exists $x_{i\ell} \in \supp(u)$ such that $x_{i\ell} \notin \{ x_{1n}, x_{n1} \}$. Hence, there exists $x_{jm}$ such that $i<j$ and $\ell < m$, or $j<i$ and $m<\ell$. Without loss of generality, suppose there exists $x_{jm}$ such that $i<j$ and $\ell < m$. Set $a = x_{i\ell}$, $b=x_{im}$, $c=x_{j\ell}$, and $d=x_{jm}$. We note that the relations implies that the Poisson subalgebra of $\cO(M_n)$ generated by $\{a,b,c,d\}$ is isomorphic to $\cO(M_2)$. Write
\[ u = \mu_1 a + \mu_2 b + \mu_3 c + \mu_4 d + \bx \]
where $\mu_i \in \kk$, $\mu_1 \neq 0$, and $\bx$ is a degree one element in $\cO(M_n)$ such that $a,b,c,d \notin \supp(\bx)$. Now
\[
\{ u, d \}
    = \mu_1 \{ a,d \} + \mu_2 \{b,d\} + \mu_3 \{ c,d \} + \{\bx, d\}
    = 2\mu_1 bc + \mu_2 bd + \mu_3 cd + \{\bx,d \}.
\]
Since $b$ and $d$ are in the same column, then $bd \notin \supp(\{\bx,d\})$. Similarly, $cd \notin \supp(\{\bx,d\})$. Furthermore, $\{a,d\}$ is the unique commutator of degree one elements such that $bc \in \supp(\{a,d\})$. Hence, $bc \notin \supp(\{\bx,d\})$.

Since $u$ is Poisson normal, then there exists a degree one element $v \in \cO(M_n)$ such that $\{u,d\} = uv$. Write
\[ v = \nu_1 a + \nu_2 b + \nu_3 c + \nu_4 d + \by \]
for $\nu_i \in \kk$ and $\by$ a degree one element in $\cO(M_n)$ such that $a,b,c,d \notin \supp(\by)$. Then
\begin{align*}
uv &= (\mu_1 a + \mu_2 b + \mu_3 c + \mu_4 d + \bx)(\nu_1 a + \nu_2 b + \nu_3 c + \nu_4 d + \by) \\
    &= (\mu_1\nu_1 a^2 + \mu_2\nu_2 b^2 + \mu_3\nu_3 c^2 + \mu_4\nu_4 d^2) 
    + (\mu_1\nu_2 + \mu_2\nu_1)ab + (\mu_1\nu_3 + \mu_3\nu_1)ac \\
    &\qquad + (\mu_1\nu_4 + \mu_4\nu_1)ad + (\mu_2\nu_3 + \mu_3\nu_2)bc
    + (\mu_2\nu_4 + \mu_4\nu_2)bd + (\mu_3\nu_4 + \mu_4\nu_3)cd + \bx\by. 
\end{align*}
Observe that the set of monomial summands appearing above is disjoint from $\supp(\bx\by)$.
Comparing coefficients of $uv$ and $\{u,d\}$ we see that $\mu_i\nu_i=0$ for $i=1,\hdots,4$. Since $\mu_1 \neq 0$ by hypothesis, then $\nu_1=0$. Thus, the above reduces to
\begin{align*}
uv &= (\mu_2\nu_2 b^2 + \mu_3\nu_3 c^2 + \mu_4\nu_4 d^2) 
    + (\mu_1\nu_2)ab + (\mu_1\nu_3)ac + (\mu_1\nu_4)ad \\
    &\qquad + (\mu_2\nu_3 + \mu_3\nu_2)bc
    + (\mu_2\nu_4 + \mu_4\nu_2)bd + (\mu_3\nu_4 + \mu_4\nu_3)cd + \bx\by. 
\end{align*}
Since $a$ and $d$ are in different rows and in different columns, it follows that $ab,ac,ad \notin \supp(\{u,d\})$.
Consequently, $\nu_2=\nu_3=\nu_4=0$. However, this means that the coefficient of $bc$ in $uv$ is $0$ which contradicts the fact that the coefficient of $bc$ in $\{u,d\}$ is $2\mu_1 \neq 0$. 
\end{proof}

In fact, one could show more using \cite[Theorem 3.4]{Giso2} and establish a normalizing sequence of degree one normal elements. We will leave that to the interested reader. 
Next we determine all Poisson reflections of $\cO(M_n)$.

\begin{lemma}
\label{lem.mat1}
Let $g$ be a Poisson reflection of $\cO(M_n)$.
For all $k<n$ and $\ell > 1$, $g$ fixes elements of the form $x_{1k}$, $x_{k1}$, $x_{\ell n}$, and $x_{n \ell}$.
\end{lemma}
\begin{proof}
Since $g$ is a Poisson reflection, there is a Poisson normal element $y$ such that $g(y)=\tornado y$ for some root of unity $\tornado \neq 1$. By Lemma \ref{lem.mnormal}, $y=\mu x_{1n} + \mu' x_{n1}$ for some $\mu,\mu' \in \kk$. Hence,
\[ \Tr_{\cO(M_n)}(g,t) = \frac{1}{(1-\tornado t)(1-t)^{n^2-1}}.\]
Set $M=\cO(M_n)/(y)$ and let $\overline{g}$ be the induced automorphism on $M$. Since $\overline{g}$ is the identity on $M$, then $g(x_{ij})=x_{ij}+\alpha_{ij} y$, $\alpha_{ij} \in \kk$, for all $1\le i,j\le n$.

Now let $k<n$. We have,
\[(x_{1k} + \alpha_{1k} y)g(x_{1n})
    = g(x_{1k}x_{1n})
    = g(\{x_{1k},x_{1n}\})
    = \{g(x_{1k}),g(x_{1n})\}
    = \{ x_{1k} + \alpha_{1k} y, g(x_{1n}) \}
    = x_{1k} g(x_{1n}).\]
The last equality holds since $x_{1n}$ and $x_{n1}$ Poisson commute with $y$ and $g(x_{1n}),g(x_{n1}) \in \Span_{\kk}\{x_{1n},x_{n1}\}$. Thus, $\alpha_{1k}=0$. Similarly, for $\ell > 1$,
\[ g(x_{1n})(x_{\ell n} + \alpha_{\ell n} y)
    = g(x_{1n}x_{\ell n})
    = g(\{x_{1n},x_{\ell n}\})
    = \{g(x_{1n}),g(x_{\ell n})\}
    = \{ g(x_{1n}), x_{\ell n} + \alpha_{\ell n} y \}
    = g(x_{1n})x_{\ell n}.\]
A similar argument applies to $x_{k1}$ and $x_{n\ell}$ with $x_{n1}$ in place of $x_{1n}$. 
\end{proof}

\begin{lemma}
\label{lem.mat2}
Let $g$ be a Poisson reflection of $\cO(M_n)$. Then
$g(x_{1n})=\mu x_{n1}$ and $g(x_{n1})=\nu x_{1n}$
for some $\mu,\nu \in \kk^\times$.
\end{lemma}
\begin{proof}
By the proof of Lemma \ref{lem.mat1}, there exists $\mu,\mu',\nu,\nu' \in \kk$ such that $g(x_{1n})=\mu x_{1n} + \mu' x_{n1}$ and $g(x_{n1})=\nu x_{1n} + \nu' x_{n1}$. We have $\{g(x_{11}),g(x_{nn})\}=2g(x_{1n})g(x_{n1})$.
Since $x_{1n}^2$ and $x_{n1}^2$ are not summands of $\{g(x_{11}),g(x_{nn})\}$, then $\mu\nu=0$ and $\mu'\nu'=0$. Thus, after relabeling, there exists $\mu,\nu \in \kk^\times$ such that one of the following hold:
\begin{enumerate}[label=(\Roman*)]
\item $g(x_{1n})=\mu x_{1n}$, $g(x_{n1})=\nu x_{n1}$,
\item $g(x_{1n})=\mu x_{n1}$, $g(x_{n1})=\nu x_{1n}$.
\end{enumerate}

Suppose $g$ is of Type I. Then both $\mu$ and $\eta$ are eigenvalues for $g$. Hence, if $g$ is a Poisson reflection, then either $\mu=1$ or $\nu = 1$. Suppose $\nu = 1$. An identical argument applies to the case $\mu=1$. 


By Lemma \ref{lem.mat1}, $g(x_{11})=x_{11}$ and $g(x_{nn})=x_{nn}$. Hence, 
\[ \mu x_{1n}x_{n1} 
    = g(x_{1n}x_{n1})
    = g(\{x_{11},x_{nn}\})
    = \{g(x_{11}),g(x_{nn})\} 
    =  \{x_{11},x_{nn}\} 
    = x_{1n}x_{n1}.\]
This implies $\mu=1$, a contradiction to the fact that $g$ has an eigenvector in ${\rm Span}_\kk\{x_{1n},x_{n1}\}$, which has eigenvalue $\neq 1$ by Lemma \ref{lem.mnormal}. Hence, $g$ is of Type II.
\end{proof}

\begin{theorem}
\label{thm.mat}
If $n>2$, then $\cO(M_n)$ has no Poisson reflections.
\end{theorem}
\begin{proof}
Let $g$ be  a Poisson reflection of $\cO(M_n)$. By Lemma \ref{lem.mat2}, $g(x_{1n})=\mu x_{n1}$ and $g(x_{n1})=\nu x_{1n}$ for some $\mu,\nu\in \kk^\times$. The submatrix $N=\begin{pmatrix}0 & \mu \\ \nu & 0\end{pmatrix}$ of $g$ has eigenvalues $\pm\sqrt{\mu\nu}$ but exactly one of these must be $1$, so $\mu\nu=1$, so $\nu=\mu\inv$, which also implies the nontrivial eigenvalue of $g$ is $-1$. The elements corresponding to the eigenvectors of $N$ are $x_{1n}+\mu x_{n1}$ and $x_{1n}-\mu x_{n1}$.

Since $g$ is a reflection, there is a Poisson normal element $y$ such that $g(y)=\tornado y$ for some root of unity $\tornado \neq 1$. By Lemma \ref{lem.mnormal}, and after factoring out the leading coefficient, we may write $y=x_{1n} + \mu x_{n1}$. 

By Lemma \ref{lem.mat1}, $g$ fixes $x_{12}$ and $x_{2n}$. Then, with the notation as in that lemma,
\begin{align*}
2x_{1n}x_{22}
    = \{x_{12},x_{2n} \} 
    = \{ g(x_{12}),g(x_{2n}) \}
    = g( \{ x_{12}, x_{2n} \})
    = g( 2x_{1n}x_{22} )
    = 2\mu x_{n1} (x_{22} + \alpha_{22}y).
\end{align*}
However, $x_{1n}x_{22}$ is not a summand of the right-hand side, a contradiction. 
\end{proof}

We now consider the $n=2$ case.

\begin{example}
\label{ex.mat2}
Let $A=\cO(M_2)$ and $g$ be a Poisson reflection of $A$. By Lemma \ref{lem.mat1} and the proof of Theorem \ref{thm.mat} there is some $\mu \in \kk^\times$ such that 
\[  g(a) = a, \quad g(b) = \mu c, \quad g(c) = \mu\inv b, \quad g(d)=d.\]
Thus, the fixed ring is $A^{\grp{g}}=\kk[a,b+\mu c,bc,d]$
with Poisson bracket
\begin{align*}
&\{a,b+\mu c\} = a(b+\mu c)	&	&\{a,bc\}=2a(bc)	&	&\{a,d\}=2(bc) \\
&\{b+\mu c,d\}=(b+\mu c)d	&	&\{bc,d\}=2(bc)d 	& 	&\{b+\mu c,bc\}=0.
\end{align*}
Moreover, let $G$ be a nontrivial Poisson reflection group of $A$. Since $G$ is generated by Poisson reflections, the above discussion implies that the action of $G$ can be restricted to the subalgebra $\kk[b,c]$. So by the STC theorem (Theorem \ref{thm.stc}), we get $\kk[b,c]^G=\kk[f,bc]$ for some homogeneous degree $n$ element $f\in \kk[b,c]_n$. Hence the fixed ring is $A^G=\kk[a,f,bc,d]$ with Poisson bracket given by  
\begin{align*}
&\{a,f\} = naf	&	&\{a,bc\}=2a(bc)	&	&\{a,d\}=2(bc) \\
&\{f,d\}=nfd	&	&\{bc,d\}=2(bc)d 	& 	&\{f,bc\}=0.
\end{align*}
We see that $A/(\{A,A\})\cong k[a,d]\oplus k[c]\oplus k[b]$, yet $A^G/(\{A^G,A^G\})\cong k[a,d]\oplus k[f]$. Thus, $A\not\cong A^{G}$ as Poisson algebras when $G$ is not trivial. 
\end{example}

\begin{theorem}\label{thm.rigM}
Let $G$ be a finite subgroup of $\GrAutp(\cO(M_n))$. If $\cO(M_n)^G\cong \cO(M_n)$ as Poisson algebras, then $G$ is trivial. 
\end{theorem}
\begin{proof}
This result follows from Theorem \ref{thm.mat} and Example \ref{ex.mat2}.
\end{proof}

\subsection{Weyl Poisson algebras}
\label{sec.WeylPoisson}
The \emph{$n$th Weyl Poisson algebra} is the Poisson algebra 
$$\scP_n=\kk[x_1,\ldots,x_n,y_1,\ldots,y_n]$$
with Poisson bracket given by 
\[
\{ x_i,y_j \} = \delta_{ij}, \quad
\{ x_i,x_j \} = \{ y_i, y_j \} = 0,
\]
for all $1\le i,j \le n$. 
Note that $\scP_n$ is not a quadratic Poisson algebra for its Poisson bracket is not homogeneous according to the standard grading on $A$.




A result of Alev and Farkas \cite{AF} shows that $\scP_1^G$ is not simple for any nontrivial group action $G$ on $\scP_1$. This implies that $\scP_1$ is rigid. One can also establish this fact by using a result of Dixmier \cite{dix}, which establishes an isomorphism between $\Autp(\cP_1)$ and $\Aut(A_1)$, along with the STC theorem (Theorem \ref{thm.stc}). More generally, by \cite[Proposition 1.1]{TIK2}, if $W,W'$ are finite subgroups of $\Autp(\cP_n)$ such that $\cP_n^W \iso \cP_n^{W'}$, then $W \iso W'$. This implies the following result. 

\begin{theorem}\label{thm.rigW}
Let $G$ be a finite subgroup of $\Autp(\scP_n)$. If $\scP_n^G\cong \scP_n$ as Poisson algebras, then $G$ is trivial. 
\end{theorem}

The following more general remark was conveyed to the authors by Akaki Tikaradze.

\begin{remark}[Tikaradze]
\label{rem.tik}
Suppose $B$ is a Poisson integral domain with finite subgroup $W$ of $\Autp(B)$, and suppose $B^W=\scP_n$ as Poisson subalgebras of $B$.
Using the terminology of \cite{BG2}, $B$ is a Poisson order over $\scP_n$. By \cite[Theorem 4.2]{BG2}, if $\fm,\fm'\in \Maxspec(\scP_n)$ lie in the same symplectic leaf of $\Spec(\scP_n)$, then $B/\fm B\cong B/\fm' B$ as finite dimensional algebras. But $\Spec(\scP_n)$ is symplectic, so $\Spec(B)\to \Spec(\scP_n)$ must be a finite covering map. On the other hand, $\Spec(\scP_n)$ is just $\kk^{2n}$, so it is simply connected. Hence, $B=\scP_n$ and $W={\id}$.
\end{remark}

In \cite{KKZ3}, Kirkman, Kuzmanovich, and Zhang considered fixed subrings of the Rees ring $H_n$ of the $n$th Weyl algebra $A_n$. They prove that $H_n^G \niso H_n$ for any reflection group $G$. We consider a Poisson version of this result.

\begin{definition}
\label{DefHn}
The \emph{homogenized $n$th Weyl Poisson algebra} is the Poisson algebra 
$$\scH_n=\kk[x_1,\ldots,x_n,y_1,\ldots,y_n,z]$$
with Poisson bracket given by 
\[
\{ x_i,y_j \} = \delta_{ij}z^2, \quad
\{ x_i,x_j \} = \{ y_i, y_j \} = \{z,-\} = 0,
\]
for all $1\le i,j\le n$.
\end{definition}
Note that $\scH_n$ is a quadratic Poisson algebra with $2n+1$ generators. 

\begin{lemma}\label{H_nResults}
\label{lem.hweyl1}
\begin{enumerate}
\item The Poisson center of $\scH_n$ is $\kk[z]$.
\item The degree one Poisson normal elements of $\scH_n$ are $\gamma z$ for some $\gamma \in \kk^\times$.
\end{enumerate}
\end{lemma}
\begin{proof}
(a) Let $c \in \pcnt(\scH_n)$ and write $c = \sum_{k=0}^m f_k y_n^k$ where each $f_k \in \kk[x_1,y_1,\hdots,x_{n-1},y_{n-1},x_n,z]$ is a polynomial. Then 
\[ 
0   = \{x_n,c\} 
    = \sum_{k=0}^m f_k\{x_n,y_n^k\}
    = \sum_{k=0}^m f_kky_n^{k-1}z^2.
\]
This implies that $f_k=0$ for all $k>0$. Thus, $c \in \kk[x_1,y_1,\hdots,x_{n-1},y_{n-1},x_n,z]$. Reversing the roles of $x_n$ and $y_n$ shows that $c \in \kk[x_1,y_1,\hdots,x_{n-1},y_{n-1},y_n,z]$. Hence, $c\in \kk[x_1,y_1,\hdots,x_{n-1},y_{n-1},z]$. An induction now implies the result.

(b) Let $u=\sum \mu_i x_i + \sum \nu_i y_i + \gamma z$ be a degree one Poisson normal element of $\scH_n$ with $\mu_i,\nu_i,\gamma \in \kk$. Suppose some $\nu_k \neq 0$. Then
\[ \{ x_k, u \}
    = \{ x_k, \sum \mu_i x_i + \sum \nu_i y_i + \gamma z \}
    = \nu_k \{ x_k,y_k \}
    = \nu_k z^2 \in (u).\]
Let $v$ be the degree one element in $\scH_n$ such that $uv = \nu_k z^2$. Write $v =\sum \mu_i' x_i + \sum \nu_i' y_i + \gamma'z$. Then the coefficient of $y_k^2$ in $uv$ is $\nu_k\nu_k'$, so $\nu_k'=0$. 
Now, the coefficient of $y_kz$ is $\nu_k\gamma'$, so $\gamma'=0$. But then the coefficient of $z^2$ in $uv$ is $0$, a contradiction. Thus, $\nu_k=0$ for all $k$. A similar argument shows that $\mu_k=0$ for all $k$. Now the result follows from (a). 
\end{proof}

We can now prove our main rigidity result regarding homogenized Weyl Poisson algebras.

\begin{theorem}
\label{thm.hweyl}
Let $g$ be a Poisson reflection of $\scH_n$.
\begin{enumerate}
    \item The automorphism $g$ is of the form $g(x_i) = x_i + a_iz$, $g(y_i) = y_i + b_iz$, and $g(z) = -z$ for $a_i, b_i \in \Bbbk$ for $i = 1, \dots, n$.  
    \item Up to a change of basis, $(\scH_n)^{\grp{g}} = \kk[x_i,y_i,z^2 : i=1,\hdots,n]$.
    \item $(\scH_n)^{\grp{g}} \niso \scH_n$.
\end{enumerate}
\end{theorem}
\begin{proof}
(a) Let $g$ be a Poisson reflection of $\scH_n$. By the definition of the trace series given in equation \eqref{eq.trace} in \S\ref{sec.background},
\[Tr_{\scH_n}(g, t) = \frac{1}{(1-t)^{2n}(1-\xi t)},\] for some root of unity $\xi \ne 1$. Since $z$ is the only Poisson normal element of degree one by Lemma \ref{H_nResults}(b), it follows that $g(z) = \lambda z$ for some scalar $\lambda$. Moreover, since $(z)$
is $g$-invariant, $g$ induces a Poisson automorphism $\overline{g}$ of $\overline{\scH_n}=\scH_n/(z)$.
Thus, by \cite[Lemma 1.7]{KKZ4}, 
\[Tr_{\overline{\scH_n}(\overline{g}, t)} = (1 - \lambda t)Tr_{\scH_n}(g, t) = (1 - t)^{-2n},\] and by \cite[Proposition 1.8]{KKZ3}, $\overline{g}$ is the identity on $\Bbbk[x_i, y_i: i = 1, \dots, n]$. That is, $g(x_i) = x_i + a_i z$, $g(y_i) = y_i + b_iz$, and $g(z) = \lambda z$, where $a_i, b_i \in \Bbbk$. Now, for $g$ to be compatible with the Poisson bracket given in Definition \ref{DefHn}, $\lambda^2 = 1$. If $\lambda = 1$, then $g$ is the identity automorphism. Thus, $g(z) = -z$ and the automorphism $g$ has the desired form.

(b) If we substitute $x_i$ with $x_i+\frac{a_i}{2}z$ and $y_i$ with $y_i+\frac{b_i}{2}z$, then $g$ is given by $g(x_i)=x_i$, $g(y_i)=y_i$, and $g(z)=-z$. Thus, $(\scH_n)^{\grp{g}} = \kk[x_i,y_i,z^2 : i=1,\hdots,n]$.

(c) Suppose $(\scH_n)^{\grp{g}} \cong \scH_n$. Consider the derived subalgebra $\{\scH_n, \scH_n\}$ of $\scH_n$. Note that every element of $\{\scH_n, \scH_n\}$ is of the form $az^2$ for some $a \in \scH_n$. Conversely, if $a \in \scH_n$, and $a'$ is the partial antiderivative of $a$ with respect to $y_i$ for $1 \le i \le n$, then
\[\{x_i,a'\}=z^2 \frac{\partial a'}{\partial y_i} = az^2.\] That is, $\{\scH_n, \scH_n\} = (z^2)$, the ideal of $\scH_n$ generated by $z^2$. Now, let $B = \Bbbk[X_i, Y_i, Z: i = 1,\hdots,n]$ with $\{X_i, Y_i\} = Z$ and $\{Z, -\} = 0$, so that $B \cong (\scH_n)^{\grp{g}}$, as Poisson algebras. Arguing as above, we find that $\{B, B\} = (Z)$. Suppose that $B \cong \scH_n$ via the isomorphism $\phi: B \to \scH_n$. 
By Lemma 6.4, $\pcnt(\scH_n) = \Bbbk[z]$. Similarly, we may say that $\pcnt(B) = \Bbbk[Z]$. Thus, the map $\phi$ restricts to an isomorphism between $\Bbbk[z]$ and $\Bbbk[Z]$, so that $\phi(Z)$ must generate a maximal ideal in $k[z]$. But 
\[\phi(Z) = \phi(\{B, B\}) =\{\phi(B),\phi(B)\}=\{\scH_n, \scH_n\} = (z^2),\]
which is a contradiction since $(z^2)$ is not a maximal ideal in $k[z]$. Thus, $\scH_n \not\cong (\scH_n)^{\grp{g}}$.
\end{proof}

\begin{corollary}
\label{cor.hweyl}
If $G$ is a nontrivial finite subgroup of the graded Poisson automorphisms of $\scH_n$ such that $(\scH_n)^G$ is a polynomial ring in $2n + 1$ variables,
then $G =\langle g\rangle$ for some Poisson reflection $g$ of ${\rm ord}(g)=2$.
\end{corollary}
\begin{proof}
By the STC theorem (Theorem \ref{thm.stc}), $G$ must be generated by Poisson reflections. Suppose, to the contrary, that $G$ contains another nonidentity Poisson reflection $g_2 \ne g_1$. It follows, by Theorem \ref{thm.hweyl} (a), that these Poisson reflections can be represented on $(\scH_n)_1$ by the $(2n + 1) \times (2n+1)$ block matrices 
\[\begin{tabular}{ cc }
$M_{g_1}$ = 
$\begin{bmatrix}
I & \overline{0}\\
\overline{v} & -1
\end{bmatrix}$, &
 $M_{g_2}$ =
 $\begin{bmatrix}
I & \overline{0}\\
\overline{u} & -1
\end{bmatrix}$,
\end{tabular}\]
where $I$ is the $2n \times 2n$ identity matrix, $\overline{v} = [~a_1 ~\dots~ a_n~ b_1~ \dots~ b_n~]$ and $\overline{u} = [a_1'~ \dots~ a_n'~ b_1'~ \dots~ b_n'~]$.
We may represent the Poisson automorphism $g_1g_2$ by the matrix product
\[\begin{bmatrix}
I & \overline{0}\\
\overline{v} - \overline{u} & 1
\end{bmatrix},\]
which has infinite order since $\overline{u} \neq \overline{v}$. Hence, $G$ contains exactly one non-identity Poisson reflection and the result follows. 
\end{proof}

\begin{theorem}\label{thm.rigidHWPA}
Let $G$ be a finite subgroup of $\GrAutp(\scH_n)$. If $\scH_n^G\cong \scH_n$ as Poisson algebras, then $G$ is trivial. 
\end{theorem}
\begin{proof}
This follows immediately from Theorem \ref{thm.hweyl} and Corollary \ref{cor.hweyl}.
\end{proof}

\subsection{Kostant-Kirillov brackets}
Let $\fg$ be an $n$-dimensional Lie algebra over $\kk$ with fixed basis $\{x_1,\hdots,x_n\}$. The Kostant-Kirillov Poisson bracket, also called the Kostant-Souriau Poisson bracket \cite{BZuni}, on the symmetric algebra $S(\fg) \iso \kk[x_1,\hdots,x_n]$ is determined by $\{x_i,x_j\}=[x_i,x_j]$. We denote this Poisson algebra as $PS(\fg)$.
We will study a homogenized version of this bracket. Given $\fg$ as above, we define $PH(\fg)$ to be the Poisson algebra on $\kk[x_1,\hdots,x_n,z]$ with Poisson bracket
\[ \{x_i,x_j\}=[x_i,x_j]z, \quad \{x_i,z\} = 0\]
for all $1\le i,j\le n$.

In the associative algebra setting, the homogenized enveloping algebra of $\fg$, denoted $H(\fg)$, is the algebra $H(\fg)$ generated by $\{x_1,\hdots,x_n,z\}$ subject to the relations
\[
    x_ix_j-x_jx_i = [x_i,x_j]z, \qquad
    x_iz-zx_i = 0,
\]
for all $1 \leq i,j \leq n$. 
This was studied in \cite{KKZ3}, wherein the authors show that $H(\fg)$ has no quasi-reflections whenever $\fg$ has no $1$-dimensional Lie ideal. Similarly, we can show that $PH(\fg)$ has no Poisson reflections. However, we will show that the correspondence is more direct than that. Note that both $PH(\fg)$ and $H(\fg)$ are connected graded and we can identify their generating spaces $PH(\fg)_1=H(\fg)_1=\Span_{\kk}\{x_1,\hdots,x_n,z\}$. 

\begin{lemma}
\label{lem.hg}
Let $V = \Span_{\kk}\{x_1,\hdots,x_n,z\}$ and let $g$ be a linear map on $V$. Then $g$ extends to an automorphism of $H(\fg)$ if and only if $g$ extends to a Poisson automorphism of $PH(\fg)$. Consequently, $g$ is a quasi-reflection of $H(\fg)$ if and only if $g$ is a Poisson reflection of $PH(\fg)$.
\end{lemma}
\begin{proof}
We abuse notation and extend the Lie bracket from $\fg$ to $V$ by assuming that $[z,-]=0$ so that the relations in 
$H(\fg)$ can be rewritten as $xy-yx=[x,y]z$ for all $x,y\in V$. Similarly, in $PH(\fg)$, the Poisson bracket is given by $\{x,y\}=[x,y]z$ for all $x,y\in V$. Suppose $g$ extends to an automorphism of $H(\fg)$. Then in $H(\fg)$, we have $[g(x),g(y)]z=g(x)g(y)-g(y)g(x)=g([x,y])g(z)$ for any $x,y\in V$. Hence in $PH(\fg)$, we have
\[
\{g(x),g(y)\}=[g(x),g(y)]z=g([x,y])g(z)=g([x,y]z)=g(\{x,y\}).
\]
So $g$ can be extended to a Poisson automorphism of $PH(\fg)$. The converse follows similarly.
\end{proof}


This next result should be compared to \cite[Lemma 6.5]{KKZ3}. The proof of part (c), however, is simplified in light of more recent results.
\begin{theorem}
\label{thm.kkbracket}
Let $\fg$ and $\fg'$ be finite dimensional Lie algebras with no $1$-dimensional Lie ideal. 
\begin{enumerate}
\item Up to scalar, $z$ is the only nonzero Poisson normal element of $PH(\fg)$ in degree one.
\item $PH(\fg)$ has no Poisson reflections.
\item If $PH(\fg) \iso PH(\fg')$, then $\fg \iso \fg'$.
\item If $G$ is a finite subgroup of $\GrAutp(PH(\fg))$ such that $PH(\fg)^G \iso PH(\fg')$, then $G$ is trivial and $\fg \iso \fg'$.
\end{enumerate}
\end{theorem}
\begin{proof}
(a) Let $b+\lambda z$ be a Poisson normal element in degree one, with $b \in \fg$ and $\lambda \in \kk$. Suppose $b \neq 0$. Let $x \in \fg$, then $\{x,b+\lambda z\}=(b+\lambda z)(y+\lambda' z)$ for some $y \in \fg$ and $\lambda' \in \kk$. But $\{x,b+\lambda z\} = \{x,b\} \in (\kk[x_1,\hdots,x_n]_1)z$. Thus $by=0$, so $y=0$. Moreover, $\lambda\lambda'=0$ so $\lambda=0$ or $\lambda'=0$. In either case we get $\{x,b\}=\lambda'bz$, so $[x,b]=\lambda b$ and $b$ generates a $1$-dimensional Lie ideal, a contradiction.

(b) This follows from Lemma \ref{lem.hg} and \cite[Lemma 6.5 (d)]{KKZ3}.

(c) Write $PH(\fg')=\kk[x_1',\hdots,x_n',z]$. 
We abuse notation and use $z$ for the homogenizing variable in both algebras.
Since both $PH(\fg)$ and $PH(\fg')$ are $\NN$-graded Poisson algebras, there exists a graded isomorphism $\phi:PH(\fg) \to PH(\fg')$ \cite[Theorem 4.2]{GXW}. By (a), $z$ is the only degree one Poisson normal element in both $PH(\fg)$ and $PH(\fg')$. Hence, $\phi(z)=\mu z$ for some $\mu \in \kk^\times$. Without loss of generality, we may assume that $\mu=1$ and $\phi(z)=z$, by applying a graded Poisson automorphism of $PH(\fg')$.

Thus, we can find a bijective $\kk$-linear map $\tau:\fg \to \fg'$ and a linear functional $\chi: \fg\to \kk$ such that $\phi(x_i)=\tau(x_i)+\chi(x_i)z$ for each $1\le i\le n$. It remains to show that $\tau$ is a Lie algebra morphism. We have
\begin{align*}
\tau(\{x_i,x_j\})z
    &= \phi(\{x_i,x_j\})z - \chi(\{x_i,x_j\})z^2 \\
    &= \phi(\{x_i,x_j\}z) - \chi(\{x_i,x_j\})z^2 \\
    &= \{\phi(x_i),\phi(x_j)\}z - \chi(\{x_i,x_j\})z^2 \\
    &= \{\tau(x_i)+\chi(x_i)z,\tau(x_j)+\chi(x_j)z\}z - \chi(\{x_i,x_j\})z^2 \\
    &= \{\tau(x_i),\tau(x_j)\}z - \chi(\{x_i,x_j\})z^2.
\end{align*}
Thus $\chi(\{x_i,x_j\})=0$ and $\tau(\{x_i,x_j\})=\{\tau(x_i),\tau(x_j)\}$ for all $1\le i,j\le n$.

(d) It follows from the STC theorem (Theorem \ref{thm.stc}) and (b) that $G$ is trivial.  Now (c) implies that $\fg \iso \fg'$.
\end{proof}

The next example show that Theorem \ref{thm.kkbracket} does not apply when the Lie algebra $\fg$ has a $1$-dimensional Lie ideal.

\begin{example}
Let $\fg=\Span_\kk\{x_1,x_2\}$ be the 2-dimensional non-abelian Lie algebra with bracket $[x_1,x_2]=x_2$. Note that $(x_2)$ is a 1-dimensional Lie ideal. Let $g$ be a Poisson reflection of $PH(\fg)$. A computation shows that $g$ is given by
\[
g(x_1) = x_1 + \alpha x_2, \quad
g(x_2) = \tornado x_2, \quad
g(z) = z,
\]
where $\tornado \in \kk^\times$ is a  primitive $m$th root of unity $\tornado \neq 1$ and $\alpha \in \kk$. After a linear change of basis, we may assume $g(x_1)=x_1$, $g(x_2)=\tornado x_2$, and $g(z)=z$. It follows that $PH(\fg)^{\grp{g}} = \kk[x_1,x_2^m,z]$ and that $PH(\fg)^{\grp{g}} \iso PH(\fg)$.
\end{example}


\section{Reflections of Poisson universal enveloping algebras}
\label{sec.UEA}

For a Poisson algebra $A$, the \emph{Poisson universal enveloping algebra} $U(A)$ is an associative $\kk$-algebra that is universal with respect to the existence of an algebra embedding $m:A \to U(A)$ and a Lie algebra map $h: A\to U(A)$ satisfying certain compatibility conditions.
The Poisson universal enveloping algebra $U(A)$ basically transfers the Poisson structure of $A$ to the algebra structure of $A$ from a representation-theoretic point of view. 

Write $m_a := m(a)$ and $h_a:=h(a)$ for all $a \in A$. According to \cite{UU}, $U(A)$ is generated by $m_A$ and $h_A$ subject to the following relations for all $x,y \in A$,
\begin{align*}
m_{xy}&= m_xm_y, &
m_{\{x,y\}}&=h_xm_y-m_yh_x, \\ 
h_{xy}&=m_yh_x +m_xh_y, &
h_{\{x,y\}}&=h_xh_y-h_yh_x.
\end{align*}
with $1_{U(A)}=m_{1_A}$. Generally speaking, the algebra $U(A)$ could be very complicated and highly noncommutative. In fact, 
$\GKdim(U(A))$ is twice the $\GKdim(A)$ for any quadratic Poisson algebra, $A$.

\begin{lemma}
\label{lem.envaut}
Let $A$ be a Poisson algebra and let $g \in \Autp(A)$. Then $g$ extends to an algebra automorphism of $U(A)$ by setting $g \cdot m_a = m_{g(a)}$ and $g \cdot h_a = h_{g(a)}$ for all $a\in A$.
\end{lemma}
\begin{proof}
To prove our claim, it suffices to verify that $g$ respects the defining relations of $U(A)$. This can be computed as shown here:
\begin{align*}
g \cdot m_{xy}
    &= m_{g(xy)} = m_{g(x)g(y)} = m_{g(x)}m_{g(y)} = (g \cdot m_x)(g \cdot m_y) \\
g \cdot h_{\{x,y\}}
    &= h_{g(\{x,y\})} = h_{\{g(x),g(y)\}}
    = h_{g(x)}h_{g(y)}-h_{g(y)}h_{g(x)}
    = (g \cdot h_x)(g \cdot h_y) - (g \cdot h_y)(g \cdot h_x) \\
g \cdot h_{xy}
    &= h_{g(xy)}
    = h_{g(x)g(y)}
    = m_{g(y)}h_{g(x)} + m_{g(x)}h_{g(y)}
    = (g \cdot m_y)(g \cdot h_x) + (g \cdot m_x)(g \cdot h_y) \\
g \cdot m_{\{x,y\}} 
    &= m_{g(\{x,y\})}
    = m_{\{g(x),g(y)\}}
    = h_{g(x)}m_{g(y)}-m_{g(y)}h_{g(x)} 
    = (g \cdot h_x)(g \cdot m_y) - (g \cdot m_y)(g \cdot h_x) \\
g \cdot m_1 &= m_{g(1)} = m_1 = 1.\qedhere
\end{align*}
\end{proof}

In light of Lemma \ref{lem.envaut}, we abuse notation and refer to $g$ as an automorphism of both $A$ and $U(A)$. The next example illustrates the difficulty in transferring invariant-theoretic information on $g$ from $A$ to $U(A)$.

\begin{example}
Let $A=\kk[x]$ be the Poisson algebra with trivial Poisson bracket. Then $U(A) = \kk[X,Y]$ where $X = m_x$ and $Y=h_x $. Let $g$ be the Poisson automorphism of $A$ determined by $g(x)=-x$. Extending $g$ to $U(A)$ we have $g(X)=-X$ and $g(Y)=-Y$. Hence, while $g$ is a Poisson reflection of $A$, it is not a reflection of $U(A)$. Moreover, we have $A^{\grp{g}} =\kk[x^2] \iso A$ so $U(A^{\grp{g}}) \iso U(A)$, but 
$U(A)^{\grp{g}} = \kk[X^2,XY,Y^2]/(X^2Y^2-XY)$. Hence, $U(A)^{\grp{g}} \niso U(A^{\grp{g}})$.
\end{example}

The next example illustrates a more general phenomenon, the fact that for quadratic skew-symmetric Poisson algebras, Poisson reflections never correspond to reflections of their Poisson enveloping algebras (Theorem \ref{thm.qpr}).

\begin{example}
Let $A=\kk[x,y]$ with $\{x,y\}=p xy$ for some $p\in \kk^\times$. 
Set $x_1 = m_x$, $x_2 = h_x$,
$y_1 = m_y$, and $y_2 = h_y$.
Then by direct computation or using the tools in \cite{LWZ3}, $U(A)$ is generated by $\{x_1,x_2,y_1,y_2\}$ with relations
\begin{align*}
&[x_1,x_2]=[y_1,y_2]=[y_1,x_1]=0, \\
&[x_2,y_1]=[y_2,x_1]=px_1y_1, \quad
[y_2,x_2]=px_1y_2 + (px_2+p^2 x_1)y_1.
\end{align*}

By Proposition \ref{prop.n2} (b), a reflection $g$ of $A$ either has the form
\begin{enumerate}
\item $g(x)=\tornado x$ and $g(y) = y$, or
\item $g(x) = x$ and $g(y)=\tornado y$,
\end{enumerate}
for some primitive root of unity $\tornado\neq 1$. Hence, in either case, we have
\[ \Tr_A(g,t) = \frac{1}{(1-\tornado t)(1-t)}.\]

Assume we are in case (b). Case (a) is similar. The action of $g$ extends to $U(A)$ by
\[ g(x_i)= x_i, \quad g(y_i) = \tornado y_i, \quad \text{ for $i=1,2$}.\]
From the defining relations, we see that $(y_1,y_2)$ is a normalizing sequence in $U(A)$. Further $\overline{U(A)} = U(A)/(y_1,y_2) \iso \kk[x_1,x_2]$.
Thus, applying \cite[Lemma 1.7]{KKZ4} we have
\[ 
\Tr_{U(A)}(g,t) 
    = \frac{\Tr_{\overline{U(A)}}(g,t)}{(1-\tornado t)^2}
    = \frac{1}{(1-t)^2(1-\tornado t)^2}.
\]
Hence, $g$ is not a reflection of $U(A)$.
\end{example}

Most quadratic Poisson algebras we consider appear as semiclassical limits of quantum polynomial rings. The next proposition is a reasonable converse to this.

\begin{lemma}
\label{lem.qpr}
Let $A=\kk[x_1,\hdots,x_n]$ be a quadratic Poisson algebra.
Then $U(A)$ is a quantum polynomial ring of global dimension $2n$, with Hilbert series $(1-t)^{-2n}$
\end{lemma}
\begin{proof}
By definition, $A$ is a connected graded Poisson algebra. Thus, $U(A)$ inherits the connected grading from $A$, which is minimally generated by the degree one elements $m_{A_1}$ and $h_{A_1}$ \cite[Proposition 2.2 (iii)]{BZuni}. Thus, $U(A)$ is noetherian by \cite[Proposition 9]{OH3}. Furthermore, the Hilbert series of $U(A)$ is $(1-t)^{-2n}$ by \cite[Proposition 2.1 (iv)]{BZuni}.
Finally, the global dimension condition follows from \cite[Corollary 5.6]{LWZ4}.
\end{proof}

\begin{lemma}
\label{lem.normseq}
For any Poisson normal element $a\in A$, we have that $(m_a,h_a)$ is a normalizing sequence in $U(A)$.
\end{lemma}
\begin{proof}
This follows from direct computation using the relations for $U(A)$. Let $x \in A$.
Since $a$ is a Poisson normal element of $A$, then $\{x,a\}=ay$ for some $y \in A$. Thus,
\begin{align*}
m_a m_x &= m_{ax} = m_{xa} = m_x m_a, \\
m_a h_x &= h_x m_a - m_{\{x,a\}}
    = h_x m_a - m_{ya}
    = h_x m_a - m_y m_a
    = (h_x-m_y)m_a.
\end{align*}
It follows that $m_a$ is normal in $U(A)$. 
We further have that
\begin{align*}
h_a h_x &= h_x h_a + h_{\{a,x\}}
    = h_x h_a + h_{ya}
    = h_x h_a + m_x h_a + m_a h_x
    = (h_x + m_x)h_a + m_a h_x, \\
h_a m_x &= m_x h_a + m_{\{x,a\}}
    = m_x h_a + m_{ya}
    = m_x h_a + m_y m_a.
\end{align*}
Thus, $\overline{h_a}$ is normal in $U(A)/(m_a)$.
\end{proof}

\begin{theorem}
\label{thm.qpr}
Let $A$ be a quadratic Poisson algebra and $g$ a Poisson reflection of $A$. Then we have
\[ \Tr_{U(A)}(g,t) = \left( \Tr_A(g,t) \right)^2.\]
Consequently, $g$ is not a quasi-reflection of $U(A)$.
\end{theorem}
\begin{proof}
Since $g$ is a reflection, there is a basis $y_1,\hdots,y_n$ of $A$ such that $g(y_1)=\tornado y_1$ for a root of unity $\tornado \neq 1$ and $g(y_i)=y_i$ for all $i>1$. Equivalently, \[\Tr_A(g,t)=\frac{1}{(1-t)^{n-1}(1-\tornado t)}.\]

By Lemmas \ref{lem.normal} and \ref{lem.subalg},
$y_1$ is a Poisson normal element of $A$. 
Hence, by Lemma \ref{lem.normseq}, $(m_{y_1},h_{y_1})$ is a normalizing sequence in $U(A)$. Since $U(A)$ is a quantum polynomial ring and $\deg(m_{y_1})=\deg(h_{y_1}) = 1$, then $\overline{U}=U(A)/(m_{y_1},h_{y_1})$ is also a quantum polynomial ring with Hilbert series $H_{\overline{U}}=(1-t)^{-2(n-1)}$. 

We compute $\Tr_{U(A)}(g,t)$ using \cite[Lemma 1.7]{KKZ4}.
As $g$ acts trivially on $\overline{U}$, then $\Tr_{\overline{U}}(g,t) = H_{\overline{U}}(t)$. Then we have 
\[ \Tr_{U(A)}(g,t) 
    = \frac{Tr_{\overline{U}}(g,t)}{(1-\tornado t)^2}
    = \frac{1}{(1-t)^{2(n-1)}(1-\tornado t)^2}
    = \left( \Tr_A(g,t) \right)^2.\qedhere\]
\end{proof}

Therefore, one sees that that Poisson reflection groups of $A$ do not naturally extend to quasi-reflection groups of $U(A)$.

\section{Some remarks and questions}\label{sect:rm}
In this last section, we provide some remarks and questions for future projects regarding invariant theory for quadratic Poisson algebras. 

Many classes of quadratic Poisson algebras are proved to be graded rigid in Theorem \ref{thm.m2}. It is interesting to ask whether they are indeed rigid, that is,  can we replace the finite subgroup of graded Poisson automorphisms with any finite subgroup of all Poisson automorphisms?

\begin{question}
Which quadratic Poisson algebra listed in Theorem \ref{thm.m2}  are indeed rigid? Is there a quadratic Poisson algebra that is graded rigid but not rigid?  
\end{question}

In invariant theory many interesting fixed subrings are Gorenstein rings. A famous result of Watanabe \cite[Theorem 1]{watanabe} states that the fixed subring $\kk[x_1,\ldots, x_n]^G$ is Gorenstein if $G$ is a finite subgroup of $SL_n(\kk)$ and the inverse statement also holds if $G$ has no reflections. Since we have shown that many examples of quadratic Poisson algebras do not have Poisson reflections, we would like to know whether they have Gorenstein fixed Poisson subrings.

\begin{question}
Which quadratic Poisson algebras have a graded Poisson automorphism group that intersects $SL(\kk)$ nontrivially?
\end{question}

In Section \ref{sec.rigidity}, we consider the homogeneous Jacobian bracket on $\kk[x,y,z]$ given by the potential $f_{p,q}$ and show that its fixed subring $\kk[x,y,z]^G$ under any finite subgroup $G$ of $GL_3(\kk)$ does not have Jacobian bracket. This gives an example of a unimodular Poisson algebra whose Poisson fixed subring is not unimodular. 

\begin{question}
For a quadratic Poisson algebra that is unimodular, under what condition is its fixed subring also unimodular?   
\end{question}

In Section \ref{sec.UEA} we show that Poisson automorphisms of $A$ can be extended naturally to automorphisms of the Poisson universal enveloping algebra $U(A)$. However under this extension those Poisson reflections of $A$ do not correspond to quasi-reflections of $U(A)$.  

\begin{question}
Let $A$ be a quadratic Poisson algebra, does the Poisson universal enveloping algebra $U(A)$ contain any nontrivial quasi-reflections?
\end{question}

In noncommutative invariant theory, group actions are generalized by Hopf actions. Although many families of noncommutative rings admit few group actions, they do have a lot of quantum symmetries when looking at arbitrary Hopf actions. We would like to know what kind of role does quantum symmetry play in deformation quantization. 

\begin{question}
Does a suitable notion of Hopf action exist on quadratic Poisson algebras? If so, can their fixed Poisson subrings be described?  
\end{question}

\bibliographystyle{plain}

\end{document}